\documentclass[reqno]{article}
\usepackage{preamble}
\addtotheorempostheadhook{\leavevmode\phantomsection}
\title{\sc Reconstruction of slices of cubes}

\author{\normalsize Marie-Charlotte Brandenburg, Juliana Curtis, Chiara Giardino, Chiara Meroni, Arnau Padrol}

\date{}

\begin{document}
\maketitle

\begin{abstract}
We prove that the combinatorial type of a slice of the cube by a hyperplane is completely determined by its graph, within the class of cube slices. We also show that such a result cannot extend to slices of arbitrary simple polytopes by constructing a simple polytope that has two combinatorially distinct slices with isomorphic graphs.
\end{abstract}

\section{Introduction}

The combinatorial type of any $3$-dimensional polytope is completely determined by its graph~\cite{whitney1932}. In contrast, starting in dimension $4$ it is not hard to find combinatorially distinct polytopes with the same graph. For example, all neighborly polytopes on $n$ vertices in dimension~$d\geq 4$ share the same graph, the complete graph, even though there are superexponentially many combinatorial types~\cite{Pad13}. This shows in particular that not even the dimension is determined by the graph.

There are, however, special families of polytopes whose complete combinatorial structure can be read from their graph. A remarkable case are  simple polytopes. Blind and Mani showed that the combinatorial type of a simple polytope is completely determined by its graph~\cite{BlindMani1987}, and later Kalai provided a very simple and elegant proof~\cite{Kalai1988}. In~\cite{Joswig00:ReconstructNonSimplePolytope}, Joswig extended this proof and showed that non-simple polytopes could be recovered from their graph provided that some additional local information on the neighborhoods of the vertices, the edge labeled vertex figure, is given.

Many similar results are better formulated in terms of class reconstructibility, in which only a constrained family of polytopes is considered \cite{PVS22:Reconstructibility}. 
There are actually two slightly different versions of the problem that have been used in the literature: Within the class, one may ask whether the graph uniquely determines the combinatorial type of the polytope or, slightly stronger, if it completely determines its face lattice on the given labeled vertex set (i.e., which subsets of vertices of the graph form a face). These notions can differ if there are automorphisms of the graph that do not extend to an automorphism of the polytope, as we will see below (e.g. in \cref{ex:ambiguous_recovery}).

More precisely, for a given family of polytopes $\mathcal P$, 
\begin{enumerate}
    \item We say that $P\in \mathcal P$ is \defn{class reconstructible} within $\mathcal P$ if, for all $Q\in\mathcal P$,
    \[
    G(P)\cong G(Q) \qquad\Longrightarrow\qquad P\cong Q.
    \]
    Thus, under the assumption that $P$ belongs to $\mathcal P$, its graph determines the isomorphism class of its face lattice.
    The whole class $\mathcal P$ is called \defn{class reconstructible} if this holds for every $P \in  \mathcal{P}$.

    \item We say that $P \in \mathcal P$ is \defn{labeled class reconstructible} within $\mathcal P$ if, 
    for all $Q\in\mathcal P$, every bijection $\verts(P)\to \verts(Q)$ inducing a graph isomorphism between $G(P)$ and $G(Q)$ extends to an isomorphism of face lattices. That is, 
   under the assumption that $P\in\mathcal P$, the subsets of $\verts(P)$ that are vertex sets of faces of $P$ are uniquely determined by the isomorphism type of its graph.
    The whole class $\mathcal P$ is called \defn{labeled class reconstructible} if this holds for every $P \in  \mathcal{P}$.
\end{enumerate}

Note that, with these definitions, the dimension is not regarded as part of the input. Thus, when we say that a polytope is (class) reconstructible from its graph, we mean that the graph alone determines both its dimension and its face lattice. Of course, the dimension can be added as a restriction by limiting to a family~$\mathcal P$ consisting of only $d$-polytopes.

With this convention, the classical reconstruction results for simple polytopes \cite{BlindMani1987,Kalai1988}, as well as several extensions to polytopes close to being simple \cite{pinedavillavicencio2018excessdegreepolytope,pinedavillavicencio2018polytopesclosesimple}, are naturally viewed as labeled class reconstruction results. The same is true for zonotopes \cite{BEZ90:ReconstructZonotope}, matroid base polytopes \cite{PVS22:Reconstructibility}, and, from the dual graph, capped cubical
polytopes \cite{Joswig00:ReconstructNonSimplePolytope}.

Under the classical convention, where $d$ is given, some of these results become unrestricted reconstruction statements. For instance, a $d$-polytope whose graph is $d$-regular is necessarily simple, and more generally several classes close to being simple can be recognized from the graph once $d$ is known. We emphasize that our convention is stronger: throughout, the dimension must itself be recovered from the graph.

In \cite{BraMer25:6Cube}, it was observed that \emph{slices, namely hyperplane sections, of cubes} of dimension up to~$6$ (and~$7$ for central slices) were completely determined by their graph, which led to the question whether the same could be true in any dimension~\cite[Question~1]{BraMer25:6Cube}. That is, if we are given the additional information that a polytope $P$ is a slice of a cube, can we reconstruct its face lattice from the graph? Our main result positively answers this open question:

\begin{theorem}\label{thm:mainthm}
The family of cube slices is class reconstructible.
\end{theorem}

It is natural to ask whether these reconstruction results extend to sections of higher codimension subspaces. Such sections include many interesting polytopes: for example, every centrally symmetric polytope is a central section of a cube. Without any restriction this is hopeless, since every polytope is a section of a simplex and simplices are sections of the cube. In fact, in \Cref{thm:lowerdimsections} we prove that graph reconstruction fails for cube sections of codimension at least four. Whether graph reconstruction is possible for cube sections of codimensions two and three remains open.

Another natural question to ask is whether this extends to slices of other families of polytopes. A classical classification result of Grünbaum implies that slices of the $d$-simplex are reconstructible from their graph (see \cref{cor:simplexslices}).
However, it is not true that one can always determine slices of a given polytope from their graphs. For example, let $P$ and $Q$ be any $(d-1)$-dimensional polytopes with the same graph but distinct combinatorial types. Let $R$ be the convex hull of $P \times \{0\}$ and $Q \times \{1\}$. Then $R$ has $P$ and $Q$ as facets, in particular as slices, which prohibits reconstruction. It is not clear to which extent the techniques in this paper can be further exploited for polytopes other than simplices and cubes. We pose this as a question.
\begin{question}\label{question:slice_of_which_poly}
	For which families of polytopes are the combinatorial types of their slices entirely determined by their graphs?
\end{question}

An obvious candidate would be to extend the results to slices of simple polytopes, as they share the same vertex-figure properties (those in \cref{prop:deg_vertices}) that we deeply exploit in our proof. However, our second main result is that \cref{question:slice_of_which_poly} has a negative answer for simple polytopes. Classical examples of polytopes with isomorphic graph but non-isomorphic face lattice do not help us since they cannot be realized as slices of a simple polytope. Here we construct two such polytopes in dimension four and prove the following.  
\begin{theorem}\label{thm:slice_simple_polytope}
    For any $d\geq 5$, slices of simple $d$-dimensional polytopes are not class reconstructible.
\end{theorem}

\paragraph{Paper overview.} This article is organized as follows. In \Cref{sec: preliminaries}, we recall basic definitions concerning polytopes, vertex-edge graphs, graph operations and slices, together with classical results from the literature showing that a polytope can be reconstructed from the edge-labeled vertex figures of its vertices.
Starting from \Cref{sec: structure of cube slices}, we focus on slices of cubes and study their combinatorial structure. We prove several structural properties of these slices and show that a slice of a cube with all vertices that are non-simple must be a Cartesian product of a cube with a hypersimplex.

The proof of \cref{thm:mainthm} has two parts. The first part is global: we recover the ambient dimension $d$ in which $P = \cube \cap H$, and we recover the maximal $c$ such that $P\cong Q\times\cube[c]$. Regarding dimension, we show how to recover it in \Cref{sec: dim}. Regarding the product decomposition, the main obstacle is that the graph $G(P)$ may present more Cartesian factors than the polytope itself, so it can suggest a decomposition with factor $\cube[t]$ for $t>c$. 

In \Cref{sec:2-faces}, we transition to the second part, which is indeed local: to recognize the correct factor $\cube[c]$, a detailed analysis of the induced cycles of $G(P)$ is required.
Then, in \Cref{sec: rec thm}, it is possible to reconstruct geometrically the edge-labeled vertex figure of each vertex of $P$ from their graph. For a diagram with the main steps of the proof, see \cref{fig:proof_scheme}.

Finally, in \Cref{sec: generalizations}, we explore possible extensions of our result. First, we consider slices of cubes with codimension greater than one. We show that, in general, there exist non-combinatorially equivalent slices of the same codimension, greater or equal than~$4$, with isomorphic vertex-edge graphs. Codimension two and three remains open. We also consider extending our results to hyperplane sections of other polytopes. We prove \cref{thm:slice_simple_polytope}, discarding an extension to slices of all simple polytopes. The extension to products of simplices remains open.

\paragraph{Usage of Large Language Models.} The polytopes appearing in \cref{ex:simple_slices_counterexample} were
found by Large Language Models (LLMs), specifically ChatGPT (GPT-5.6 Sol Pro). The truncation strategy in this example was developed through an iterative dialogue between the authors and the LLM. Once this strategy had been established, ChatGPT was used to carry out an exhaustive search for suitable candidates among $4$-dimensional polytopes with few vertices \cite{FirschingComplete4polys2020}. 

The authors also used GitHub Copilot and ChatGPT to assist with the implementation of the construction in the proof of \cref{prop:embed_and_glue}, which led to the counterexample presented in \cref{ex:simple_counterexample}. All computational results relevant to the paper were independently checked by the authors using \texttt{SageMath} and \texttt{Oscar.jl}. 

All mathematical statements, arguments, and proofs in the paper were formulated, verified, and written by the authors. The authors take full responsibility for the mathematical content of the paper.

\paragraph{Acknowledgments.} 
The authors would like to thank Christian Haase and Teemu Lundstr\"om for interesting discussions on the topic.

This research was partially conducted during the Intensive Research Program on Combinatorial Geometries and Geometric Combinatorics, held at the Centre de Recerca Matemàtica in October-November 2025, and funded by the Severo Ochoa and María de Maeztu Program for Centers and Units of Excellence in R\&D (CEX2020-001084-M), the Institute de Matemàtica de la Universitat de Barcelona, and the Spanish projects PID2022-137283NB-C21 and RED2024-153572-T of MICIU/AEI /10.13039/501100011033.
We thank all organizers and participants of the program for the beautiful interactions and ideas exchange. 

MB is supported by the SPP 2458 ``Combinatorial Synergies'', funded by the Deutsche Forschungsgemeinschaft (DFG, German Research Foundation).
CG is supported by ``the Combinatorial Polytope Theory'' (COMPOTE) project (DFG 541393733). 
CM is supported by Dr. Max R\"ossler, the Walter Haefner Foundation, and the ETH Z\"urich Foundation. AP is supported by projects PID2022-137283NB-C21 of MICIU/AEI/10.13039/501100011033 / FEDER UE, the
Spanish-German project COMPOTE (AEI PCI2024-155081-2 \& DFG 541393733), as well as the Mar{\'i}a de Maeztu program CEX2020-001084-M.

\section{Preliminaries}\label{sec: preliminaries}
	
\subsection{Polytopes and graphs}\label{subsec:polytopes_graphs}	
A polytope $P$ is the convex hull of finitely many points in a Euclidean space, we say it is a $d$-polytope if the dimension of its affine hull is $d$. Its \defn{faces} are either $P$ or the intersection of $P$ with a hyperplane that does not cut its relative interior. Any face is itself a polytope, and a $k$-face of $P$ is a face of dimension~$k$. In particular, the \defn{vertices}, \defn{edges}, and \defn{facets} of $P$ are its $0$-, $1$-, and $(d-1)$-faces, respectively. The \defn{face lattice} of $P$ is the poset of faces ordered by inclusion. Two polytopes $P$ and $Q$ are \defn{combinatorially equivalent}, denoted $P\cong Q$, if they have isomorphic face lattices.
The \defn{(vertex-edge) graph} of $P$ is the graph $G(P)$ whose \defn{vertices} are the vertices of $P$, and a pair of vertices forms an \defn{edge} if they form an edge of $P$. We denote by $N_{G(P)}(v)$ the set of neighbors of a vertex $v$ in $G(P)$, namely the set of vertices $w$ such that $[v,w]$ is an edge of $G(P)$.

Let $v$ be a vertex of a $d$-polytope $P$ and $H$ a hyperplane that separates $v$ from the other vertices of $P$, i.e. $v$ lies in one open half-space bounded by $H$ and all other vertices lie in the opposite open half-space. The \defn{vertex figure} of $P$ at $v$, denoted by $\vf_P(v)$, is the $(d-1)$-polytope $P\cap H$. There is a bijection between the $k$-dimensional faces of $P$ that contain $v$ and the $(k-1)$-dimensional faces of the vertex figure of $P$ at $v$ \cite[Proposition 2.4]{Ziegler94:LectureOnPolytopes}.
For each vertex $v$ of $P$, we define the \defn{edge labeled vertex figure} of $P$ at $v$ as the set of subsets of edges of $G$ incident to $v$ that contains, for each facet $F$ of $P$ incident to $v$, the set of edges of $F$ containing $v$. We call the family of edge labeled vertex figures of $P$ at each vertex the \defn{edge labeled vertex figures} of $P$. These concepts were introduced by Joswig in~\cite{Joswig00:ReconstructNonSimplePolytope} to reconstruct polytopes from graphs.

\begin{theorem}[{\cite[Thm.~2.3]{Joswig00:ReconstructNonSimplePolytope}}]\label{thm:edgelabeledreconstruction}
    The combinatorial type of a polytope is uniquely determined by its graph together with its edge labeled vertex figures.
\end{theorem}

\subsection{Joins and Cartesian products}
The \defn{join} of two polytopes $P\subset \R^d$ and $Q\subset\R^e$ is 
\[
P \join Q = \conv \bigl((P\times\{0\}\times\{0\}) \cup (\{0\}\times Q\times\{1\})\bigr) \subset \mathbb{R}^{d+e+1}.
\] 
Its faces are the joins $F\join G$ where $F$ and $G$ are faces of $P$ and $Q$, respectively. Therefore, the graph of the join $P\join Q$ is the graph-theoretic join $G(P)\vee G(Q)$. More precisely, if $G=(V_G,E_G)$ and $H=(V_H,E_H)$ are graphs, then $G\vee H$ is the graph with vertices $V_G\sqcup V_H$ and edges $E_G\cup E_H\cup \set{[u,v]}{u\in V_G,\,v\in V_H}$.

The \defn{Cartesian product} of two polytopes $P\subset \R^d$ and $Q\subset\R^e$ is 
\[
P\times Q = \set{(x,y)}{x\in P,\ y\in Q}\subset \R^{d+e}. 
\]
The non-empty faces of $P\times Q$ are of the form $F\times G$, where $F$ and $G$ are non-empty faces of $P$ and $Q$, respectively. In particular, the graph $G(P\times Q)$ is the \defn{Cartesian product} $G(P)\square G(Q)$ of the graphs $G(P)$ and $G(Q)$. If $G=(V_G,E_G)$ and $H=(V_H,E_H)$ are graphs, then $G\square H$ is the graph with vertices $V_G\times V_H$ and where $(u_1,v_1)\sim (u_2,v_2)$ if and only if either $u_1=u_2$ and $[v_1,v_2]\in E_H$, or $v_1=v_2$ and $[u_1,u_2]\in E_G$.

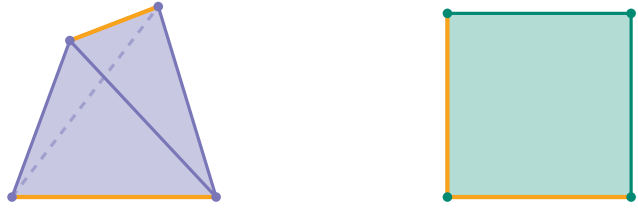
\begin{figure}[ht]
    \centering
    \begin{tikzpicture}[scale=0.9]

\definecolor{join_color}{named}{Periwinkle}
\definecolor{product_color}{named}{PineGreen}

\begin{scope}[xshift=-3.2cm]

    \coordinate (A1) at (-1.5,-0.7);
    \coordinate (A2) at ( 1.5,-0.7);
    \coordinate (B1) at (-0.65,1.6);
    \coordinate (B2) at ( 0.65,2.1);

    \draw[join_color, dashed, very thick]
        (A1) -- (B2);
    \fill[white, fill opacity=0.5]
        (A1) -- (A2) -- (B2) -- (B1) -- cycle;
    

    \fill[join_color, fill opacity=0.4]
        (A1) -- (A2) -- (B1) -- cycle;

    \fill[join_color, fill opacity=0.4]
        (A2) -- (B1) -- (B2) -- cycle;

    \draw[join_color, very thick]
        (A1) -- (A2)
        (B1) -- (B2)
        (A1) -- (B1)
        (A2) -- (B1)
        (A2) -- (B2);
        
    \draw[YellowOrange, ultra thick]
        (A1) -- (A2);

    \draw[YellowOrange, ultra thick]
        (B1) -- (B2);

    \foreach \p in {A1,A2,B1,B2}{
        \fill[join_color] (\p) circle (2pt);
    }
\end{scope}

\begin{scope}[xshift=3.2cm]

    \filldraw[
        fill=product_color,
        fill opacity=0.3,
        draw=product_color,
        very thick
    ]
        (-1.5,-0.7) --
        ( 1.2,-0.7) --
        ( 1.2, 2.0) --
        (-1.5, 2.0) --
        cycle;

    \draw[YellowOrange, ultra thick]
        (-1.5,-0.7) -- ( 1.2,-0.7)
        (-1.5,-0.7) -- (-1.5, 2.0);

    \foreach \p in {
        (-1.5,-0.7),
        ( 1.2,-0.7),
        ( 1.2, 2.0),
        (-1.5, 2.0)
    }{
        \fill[product_color] \p circle (2pt);
    }
\end{scope}

\end{tikzpicture}
    \caption{Join and Cartesian product of two copies of a segment (highlighted in orange). The graph of the join is the complete graph $K_4$, whereas the graph of the Cartesian product is a square $K_2\square K_2$.}
    \label{fig:join_cartesian}
\end{figure}
Properties of Cartesian products in graphs were investigated in \cite{Sabidussi, Vizing,ImrichPeterin}. A graph is called \defn{prime} if it cannot be presented as the product of two nontrivial graphs, that is, as the product of two graphs with at least two vertices.
\begin{theorem}(\cite{Sabidussi, Vizing})
	\label{thm:cartesian_product}
	Every connected finite graph has a prime factor decomposition with respect to the Cartesian product that is unique up to the order and isomorphism of the factors.
\end{theorem}

\subsection{Slices}
We define a \defn{slice} of a $d$-dimensional polytope $R$ to be any $(d - 1)$-dimensional polytope $P$ which can be obtained as the intersection of $R$ with an affine hyperplane~$H = \set{x\in \R^d }{ u_1 x_1 + \ldots + u_d x_d = t}$. We will use this notation throughout, and reserve $u = (u_1,\dots,u_d)$ to denote the normal vector of the intersecting hyperplane $H$.

\begin{lemma}\label{lem:slicefacedim}
	Let $R\subset \R^d$ be a polytope and $P = R\cap H$ a slice of $R$. Any $k$-face of $P$ is either a $k$-face of $R$ contained in $H$ or the intersection of $H$ with a $(k+1)$-face of $R$ (through its relative interior).
\end{lemma}
\begin{proof}
	The faces of $P$ are the non-empty intersections $G \cap H$ where $G$ is a face of $R$. 
    So let $F$ be a $k$-face of $P$ for $k \geq 0$ and suppose $F = G \cap H$ where $G$ is a face of $R$ and further assume that $\dim(G)$ is minimal among all such faces. 
    Then, there are two cases. 
    If $G \subseteq H$, then $G$ itself is a face of $P$ and $\dim(G \cap H) = \dim(G) = k$.
    If not, then since $\dim(G)$ is minimal it must be that $H$ intersects the relative interior of $G$. 
    Then $\dim(G \cap H) = \dim(G) - 1 = k$, so $\dim(G) = k + 1$.
    Thus, a $k$-face of $P$ must either be a $k$-face of $R$ contained in $H$ or the intersection of $H$ with a $(k+1)$-face of $R$. 
\end{proof}

The \defn{$d$-simplex} is the convex hull of $d+1$ affinely independent points. In its standard form it is
\[
\simp = \conv (e_1,\dots, e_{d+1}),
\]
where $e_i$ is the $i$-th standard basis vector. Its vertex-edge graph is the \defn{complete graph} on $d+1$ vertices $K_{d+1}$.

A slice of a $d$-simplex is a $(d-1)$-polytope with at most $d+1$ facets, by construction. The polars of such polytopes where characterized by Gr\"unbaum.
\begin{theorem}[{\cite[Sec.~6.1]{Grunbaum03:ConvexPolytopes}}]\label{thm:Grunbaum}
    A $d$-polytope with at most $d+2$ facets is combinatorially equivalent to 
    \[
		\simp[p] \join \bigl(\simp[{q}]\times \simp[{r}]\bigr),
		\]
    for some $p\geq -1$ and $q,r\geq 0$ satisfying $p+q+r+1=d$, where $\Delta_{-1} = \emptyset$.
\end{theorem}

\begin{corollary}\label{cor:simplexslices}
    Slices of the $d$-simplex are labeled class reconstructible.
\end{corollary}
\begin{proof}
The graph of $\simp[p] \join \bigl(\simp[{q}]\times \simp[{r}]\bigr)$ is $(K_{q+1}\square K_{r+1})\vee K_{p+1}$. This decomposition is unique (the vertices of $K_{p+1}$ are those adjacent to all other vertices, and the decomposition $(K_{q+1}\square K_{r+1})$ is unique by \cref{thm:cartesian_product}). This completely determines $p,q,r$, and thus the combinatorial type of the slice. The graph also determines the facets as subsets of its given vertex set, proving labeled class reconstructibility.
\end{proof}

The preceding classification also determines the possible local combinatorial structure, in terms of vertex figure, of slices of simple polytopes.
\begin{lemma}\label{lemma:vertex_figure}
    Let $R\subset \R^d$ be a simple polytope and $P = R\cap H$ a slice. Then, at every vertex $v$ of $P$, there exist $p\geq -1$, $q,r \geq 0$ satisfying $p+q+r+1=d-2$ such that
    \[
    \vf_P(v) \cong \simp[p] \join (\simp[q] \times\simp[r]).
    \]
\end{lemma}
\begin{proof}
    The vertex figure $\vf_P(v)$ has dimension $d-2$.
    Since $R$ is simple, exactly $d$ facets of $R$ are incident to each vertex of $R$, and exactly $d-1$ facets meet in an edge of $R$. Consequently, by \cref{lem:slicefacedim}, at most $d$ facets of $P$ are incident to $v$, and therefore $\vf_P(v)$ has at most $d$ facets.
    Thus $\vf_P(v)$ is a $(d-2)$-polytope with at most $(d-2)+2$ facets. The claim now follows from \cref{thm:Grunbaum}.
\end{proof}

\section{The structure of cube slices}\label{sec: structure of cube slices}

The \defn{standard $d$-cube} is the polytope \defn{$\cube$} obtained as the convex hull of $\{0,1\}^d$, the $d$-dimensional vectors with coordinates $0$ and $1$. It is the $d$-fold Cartesian product of the segment $[0,1] \subset \R$. 
We denote by \defn{$\cubegraph$} the vertex-edge graph of $\cube$. Note that $\cubegraph$ is a Cartesian product of $d$ edges $K_2 \square \cdots \square K_2$.
We denote by $\supp(v) = \{i \in [d] \mid v_i \neq 0\}$ the \defn{support} of a vector $v$. For a set $A \subset [d]$ we denote $e_A = \sum_{i \in A} e_i$. Occasionally, we write $e_{ij}$ as a shorthand-notation for~$e_{\{i,j\}}$, $e_{ijk}$ for~$e_{\{i,j,k\}}$, and so on. We will use this notation throughout the paper. For a hyperplane $H = u^\perp$ we say that a coordinate $i \in [d]$ is \defn{free} in $H$ if $u_i = 0$, i.e., if $x \in H$ implies $x + \lambda e_i \in H$ for all $\lambda \in \R$. Similarly, for a polytope $F$ we say that a coordinate is \defn{free} if $F$ lies in a hyperplane where the coordinate is free.

We devote the rest of the section to characterize features of \emph{slices} of cubes. 	
Throughout the paper, we will often use the highly symmetric structure of the cube to simplify arguments. 
Without loss of generality, we can make one of the following assumptions on any slice $P=\cube\cap H$ of $\cube$ :
\begin{enumerate}[label=(A\arabic*)]
    \item  \label{assumption:non_negative} The hyperplane $H$ has a non-negative normal vector $u$ so that
    \begin{equation*}
    H = \{x\in \R^d \mid u_1 x_1 + \ldots + u_d x_d = t\}, \qquad u_i\geq 0 \; \forall i \in[d].
\end{equation*}
    \item \label{assumption:through_origin} If the slice $P$ contains a vertex $v$ of $\cube$, then $v=\zero \in \verts(P)\cap \verts(\cube)$ and 
\begin{equation*}
    \qquad H = u^\perp = \{x\in \R^d \mid u_1 x_1 + \ldots + u_d x_d = 0\}.
\end{equation*}
\end{enumerate}
Both assumptions are separately valid and we will make use of them when convenient. 

\subsection{Simple and non-simple slices}
We start our study of slices of cubes focusing on the possible degrees of its vertices, to distinguish simple and non-simple slices.

\begin{proposition} \label{prop:deg_vertices}
    Let $P=H\cap \cube$ be a slice of $\cube$ by the hyperplane $H$. If $v$ is a vertex of both $P$ and $\cube$, then the degree of $v$ within $P$ is $c + ab$. Further, of the $d$ edges of $\cube$ incident to $v$, there are $c$ which are contained in $H$ while $a$ lie on one side of $H$ (above) and $b=d-a-c$ lie on the other side (below). 
\end{proposition}
\begin{proof}
    Without loss of generality, we may assume that \ref{assumption:through_origin} holds, i.e., that $H = u^\perp$ and $v=\zero$. If $H$ does not intersect the interior of $\cube$, then $P$ is a face of the cube and $H$ is a supporting hyperplane. In this case, $c = \dim P = \deg_P(v)$ and either $a$ or $b$ is zero, so the claim follows. 
    
    Otherwise, $H$ intersects the interior of the cube and, as a consequence, there must exist $i,j \in [d]$ such that $u_i u_j <0$.
    The edges of $\cube$ incident to $\zero$ are $[\zero, e_i]$ for $i\in [d]$. Then, as $\langle u, e_i\rangle = u_i$, the edge $[\zero, e_i]$ is contained in $P$ if and only if $u_i=0$; it lies above $P$ when $u_i>0$ and below $P$ when $u_i<0$. We denote the numbers of such edges by $c$, $a$, and $b=d-c-a$ respectively. See \cref{fig:cartoon_abc} for a cartoon.
    \begin{figure}[ht]
        \centering
        \begin{tikzpicture}[scale=0.6]

\definecolor{join_color}{named}{Periwinkle}
\definecolor{product_color}{named}{PineGreen}

\coordinate (O) at (0,0);

\coordinate (J1) at (2.1,-1.5);
\coordinate (Jb) at (4.4,-1);

\draw[very thick] (O) -- (J1);
\draw[very thick] (O) -- (Jb);

\node[below]       at (J1) {$e_{j_1}$};
\node[below right] at (Jb) {$e_{j_b}$};
\node at (3.3,-1.4) {$\cdots$};

\filldraw[
    fill=white,
    fill opacity=0.7,
    draw=product_color,
    very thick
]
    (-1.4,-1.25) --
    (-1.1,0.95) --
    (4.55,2.75) --
    (4.25,-0.15) --
    cycle;
\filldraw[
    fill=product_color,
    fill opacity=0.3,
    draw=product_color,
    very thick
]
    (-1.4,-1.25) --
    (-1.1,0.95) --
    (4.55,2.75) --
    (4.25,-0.15) --
    cycle;

\coordinate (I1) at (-0.25,2.8);
\coordinate (Ia) at (1.45,3);

\draw[very thick] (O) -- (I1);
\draw[very thick] (O) -- (Ia);

\node[above left]  at (I1) {$e_{i_1}$};
\node[above right] at (Ia) {$e_{i_a}$};
\node at (0.65,3.2) {$\cdots$};

\coordinate (K1) at (2.9,0.4);
\coordinate (Kc) at (3.1,1.85);

\draw[product_color!80!black, very thick] (O) -- (K1);
\draw[product_color!80!black, very thick] (O) -- (Kc);

\node[right] at (K1) {{\color{product_color!70!black} $e_{k_1}$}};
\node[right] at (Kc) {{\color{product_color!70!black} $e_{k_c}$}};
\node[right] at (3.1,1.3) {{\color{product_color!70!black} $\vdots$}};

\fill (O) circle (3pt);

\node at (4.9,2.9) {\textcolor{product_color!80!black}{$H$}};
\node at (-0.3,0) {$\zero$};

\end{tikzpicture}
        \caption{Cartoon of the local picture around a vertex (the origin $\zero$) of both $P$ and $\cube[d]$.}
        \label{fig:cartoon_abc}
    \end{figure}
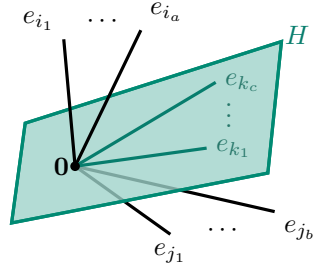
    By definition, the degree of the vertex $\zero\in P$ is the number of edges of $P$ incident to $\zero$. By \cref{lem:slicefacedim}, there are two types of edges $E$ of $P$:
    \begin{enumerate}
        \item $E$ is an edge of $\cube$ itself, namely $E\subset H$;
        \item $E = F \cap H$ where $F$ is a $2$-dimensional face of $\cube$ with $F \nsubseteq H$.
    \end{enumerate} 
    There are $c$ edges of $P$ of the first type that are incident to $\zero$.
    In the second case, $F$ is a $2$-dimensional face of $\cube$ incident to $\zero$, and thus necessarily defined as
    \[
    F = F_{i,j} = \set{x \in \R^d}{x_i, x_j \in [0,1], x_k = 0 \text{ for all } k \notin \{i,j\}}
    \]
    for some $\{i,j\}\subset [d]$. Hence, $F_{i,j}$ has vertices $\zero, e_i, e_j, e_i+e_j$, and it gives rise to an edge of $P$ only when 
    \begin{equation*}
        u_i x_i + u_j x_j = 0
    \end{equation*}
    has a solution, namely when $u_i u_j<0$.
    Since $|\{i \in [d] \mid u_i > 0 \}| = a$ and $|\{i \in [d] \mid u_i < 0 \}| = b$, there are $ab$ pairs of indices $\{i,j\}$ giving rise to faces $F_{i,j}$ that intersect $H$, and hence to edges of $P$ of the second type.
    Taking the two classes of edges together, the claim follows.
\end{proof}

As a consequence, one can classify how simple vertices of a slice can arise. This is relevant since for reconstruction purposes we only need to understand the structure around non-simple vertices.
\begin{corollary}
    \label{cor:simple_vertices}
    Let $H$ be a hyperplane that intersects the interior of $\cube$. Then, a vertex $v$ of $P = \cube\cap H$ is simple if and only if one of the following holds:
    \begin{enumerate}
        \item[(i)] $v$ is the intersection of $H$ with the relative interior of an edge of $\cube$, or
        \item[(ii)] $v$ is a vertex of $\cube$ and exactly one edge of $\cube$ incident to $v$ lies on one side of $H$.
    \end{enumerate}
\end{corollary}

\begin{proof}
    If $v$ arises as the intersection of $H$ with the relative interior of an edge of $\cube$, then \cite[Proposition 4.2]{BraMer25:6Cube} shows that $v$ is simple. 
    Otherwise, $v$ is a vertex of $\cube$. 
    Using the same notation as in the proof of \cref{prop:deg_vertices}, the vertex $v$ is simple in $P$ if and only if $c + a(d-c-a) = d-1$.
    Equivalently,
    \[
    a^2 - (d-c)a + (d-c-1) = 0,
    \]
    which has solutions $a = 1$ and $a = d-c-1$.
    In other words, $v$ is simple in $P$ if and only if there is exactly one edge of $\cube$ incident to $v$ on one side of $H$, either above (case $a=1$) or below (case $a=d-c-1)$.       
\end{proof}

\begin{remark}
    Note that \Cref{prop:deg_vertices} and \Cref{cor:simple_vertices} hold true more generally for slices of simple polytopes by a local argument.
\end{remark}

Note that if $u_i=0$ then  $\cube\cap H= Q \times [0,1]$ (up to permuting coordinates), where $Q$ is the slice by $H$ of the face of $\cube$ defined by $x_i=0$. This motivates the following definition.
\begin{definition}
	Let $P = \cube\cap H$ be a slice of $\cube$. $P$ is said to be \defn{reducible of rank $r$} if $r<d$ is the smallest integer satisfying
	\[
	P \cong Q \times \cube[{d-r}]
	\]
	where $Q$ is a slice of $\cube[r]$. Otherwise, if $r=d$, the slice is said to be \defn{irreducible}.
\end{definition}

Having $u_i=0$ allows for such a factorization, but the converse is not true in general. For example, any hyperplane cutting 4 parallel edges of the standard $3$-cube $\cube[3]$ is a combinatorial quadrilateral. The following result shows that, as soon as there is a non-simple vertex, irreducibility becomes completely determined by the zero entries of $u$.

\begin{proposition}\label{prop:u-iff-irreducible}
Let $P = \cube \cap H$ be a slice with a non-simple vertex. Under assumption \ref{assumption:through_origin}, a non-simple vertex is $\zero$ and $H = u^\perp$. Denote by $C=\set{i\in[d]}{u_i=0}$ of cardinality $c=|C|$, then $P$ is reducible of rank~$d-c$. In particular, a cube slice $P = \cube \cap H$ with $\zero$ non-simple is irreducible if and only if $u_i \neq 0$ for all $i \in [d]$.
\end{proposition}
\begin{proof}

If $u_i=0$, then for all $p\in H$ also $p+e_i\in H$, and the $i$-th coordinate is free in $H$.
In particular, if $F$ is the facet of $\cube$ given by $x_i=0$, and $Q=F\cap H$, we get
    \[
    P=\cube\cap H= (F\times [0,1])\cap H=Q \times [0,1],
    \]
where we understand the product to have the $[0,1]$ factor on the $i$-th coordinate.
This shows that	
    \[
    P \cong Q \times \cube[{c}],
    \]
where $Q$ is a slice of $\cube[d-c]$. It remains to show that $c$ is the largest possible cube factor.

Let $R$ be a polytope and $k \geq 0$ with $P \cong R \times \cube[k]$, and let $\phi$ be an isomorphism of face lattices; write $\phi(\zero) = (r, w)$ with $r$ a vertex of $R$ and $w$ a vertex of $\cube[k]$. Since the vertex figure of a product is the join of the vertex figures,
\[
  \vf_P(\zero) \cong \vf_R(r) \join \vf_{\cube[k]}(w) = \vf_R(r) \join \simp[k-1],
\]
so the graph of $\vf_P(\zero)$ contains at least $k$ vertices adjacent to all other vertices, namely those of the simplex factor. On the other hand, since $\vf_Q(\zero) \cong \simp[a-1] \times \simp[b-1]$ for some $a,b\geq 2$ because $\zero$ is non-simple, we have
\[
  \vf_P(\zero) \;\cong\; \big(\simp[a-1] \times \simp[b-1]\big) \join \simp[c-1].
\]
In a join $X \join Y$ two vertices of $X$ are adjacent if and only if they are adjacent in $X$, hence a vertex of $X$ is adjacent to all other vertices of $X \join Y$ if and only if this holds in $X$. The graph of $\simp[a-1] \times \simp[b-1]$ is $K_a \,\square\, K_b$, which has no such vertex when $a, b \geq 2$: the vertex $(i,j)$ is not adjacent to $(i',j')$ for $i \neq i'$, $j \neq j'$. Therefore the graph of $\vf_P(\zero)$ has exactly $c$ vertices adjacent to all others, and $k \leq c$. Thus $P$ has rank $d-c$.
\end{proof}

Recall from \Cref{lemma:vertex_figure} that the vertex figure of every vertex $v$ of a slice $P = \cube \cap H$ is the join of a simplex with a product of simplices. When $a,b,c$ denote the number of edges incident to $v$ which are above, below, and contained in $H$, respectively, then the proof above reveals that 
\[
\vf_P(v) \cong \simp[c-1]\join \bigl(\simp[a-1]\times\simp[b-1]\bigr).
\]

\begin{remark}
    As a consequence of \Cref{prop:u-iff-irreducible}, although different vertices of $P$ that are vertices of $\cube$ may have different degrees, they yield the same value of $c$.
    For example, in dimension $5$, the hyperplane $u^\perp = (1,1,-2,1,0)$ defines the vertex $\zero$ of degree $c+ab = 1 + 3\cdot 1$, but it also contains the non-simple vertex $v = e_{123}$. 
    With respect to this vertex, there are $a = 2, b = 2, c = 1$ incident cube-edges above, below, and contained in $H$, respectively, and $\deg_{P}(v) = c + ab = 1 + 2 \cdot 2$.
\end{remark}
	
As shown in \Cref{cor:simple_vertices}, simple vertices of a slice may arise as intersections of the hyperplane $H$ with an edge or a vertex of the cube $\cube$. However, we prove now that if one is interested only in the combinatorial type of the slice, it is not restrictive to assume that simple vertices arise from edges. We will assume this reduction throughout the rest of the paper.
\begin{proposition}\label{prop:simple_vertices_through_edges}
    Let $P = \cube \cap H$ be a slice. Then there exists $P' = \cube \cap H'$ combinatorially equivalent to $P$ and such that all the simple vertices of $P'$ are intersections of $H'$ with the relative interior of edges of the cube.
\end{proposition}
\begin{proof}
    Without loss of generality, see \ref{assumption:non_negative}, we assume that the coefficients in the equation of $H$ are non-negative. 
    We first consider the case where all coefficients are strictly positive.
    Let $v$ be a vertex of both $P$ and $\cube$, and $\supp(v) = \{i \in [d] \mid v_i = 1\}$ be the support of $v$. Because $v\in H$, we have $\sum_{i\in \supp(v)}u_i=t$.
    The edges of $\cube$ incident to $v$ are divided as follows. If $i\notin \supp(v)$, then the neighbor $x=v+e_i$ of $v$ satisfies
    \[
    \sum u_ix_i>t,
    \]
    whereas if $i\in \supp(v)$, then for the neighbor $x = v - e_i$ we have
    \[
    \sum u_ix_i<t.
    \]
    Thus, the numbers of edges of the cube incident to $v$ on the two sides of $H$ are $|\supp(v)|$ and $d-|\supp(v)|$, respectively. 
    If $v$ is simple, then $|\supp(v)|$ is either $1$ or $d-1$, as shown in \cref{cor:simple_vertices}.
    So every simple vertex of $P$ which is also a cube vertex is of one of the two forms
    \[
    v=e_i \qquad \text{or} \qquad v=\mathbf 1-e_i.
    \]
    Assume first that $v=e_i \in P$ is simple. Then, substituting into the hyperplane equation we get $u_i = t$. Hence, all other vertices of the cube that lie in $H$ must have zero in the $i$-th coordinate. Choose $\varepsilon>0$ sufficiently small and define
    \[
    H_\varepsilon = \set{x\in\mathbb R^d}{u_1x_1+\cdots+(u_i+\varepsilon)x_i+\cdots+u_dx_d=t}.
    \]
    Every cube vertex in $H$ different from $e_i$ remains in $H_\varepsilon$ because all such vertices have $x_i=0$. Moreover, by taking $\varepsilon$ sufficiently small, no cube vertex which was not on $H$ changes side with respect to the hyperplane. 
    The vertex $e_i$ itself is moved to the positive side of $H_\varepsilon$, whereas $0$ lies on the negative side. Hence $H_\varepsilon$ intersects the cube edge $[0,e_i]$ in its relative interior.
    We compare the local face structure before and after the perturbation. At the old vertex $e_i$, the cube facet $\{x_i=1\}$ cuts the section in the single point $e_i$, because
    \[
    x_i=1,\quad u_1x_1+\cdots+u_dx_d=t=u_i
    \]
    and $u_1,\dots,u_d >0$ forces all other coordinates of $x$ to be zero. Hence $\{x_i=1\}$ does not define a facet of $P$ at $e_i$. The non-redundant facets of $P$ incident to $e_i$ are precisely the $d-1$ faces
    \[
    P\cap\{x_j=0\},\qquad j\neq i.
    \]
    After the perturbation, the new point $v_\varepsilon = H_\varepsilon \cap [0,e_i]$ satisfies
    \[
    0< (v_\varepsilon)_i <1
    \qquad \text{and} \qquad
    (v_\varepsilon)_j=0 \text{ for } j\neq i.
    \]
    Thus, the non-redundant facets of $\cube\cap H_\varepsilon$ incident to $v_\varepsilon$ are again precisely
    \[
    (\cube\cap H_\varepsilon)\cap\{x_j=0\},\qquad j\neq i.
    \]
    So the simple vertex $e_i$ has just been replaced by the edge-intersection point $v_\varepsilon$, with the same incident facets. Since all other cube vertices keep the same sign relative to the hyperplane, the rest of the face lattice is unchanged. Therefore $\cube\cap H_\varepsilon$ is combinatorially equivalent to $P$.
    
    The case $v=\mathbf 1-e_i$ is analogous, and the perturbation can be obtained via
    \[
    H_\varepsilon = \set{x\in\mathbb R^d}{u_1x_1+\cdots+(u_i+\varepsilon)x_i+\cdots+u_dx_d=t+\varepsilon}.
    \]
    We can perform these perturbations vertex by vertex, since each affects a different coordinate of the hyperplane.
    This proves the result when all coefficients in the defining equation of $H$ are strictly positive.
    
    We now consider the case where not all coefficients are strictly positive. Without loss of generality, let $1,\dots,r$ be the indices of the non-zero coefficients, and $r+1,\dots,d$ the indices of the zero coefficients. Then, by \Cref{prop:u-iff-irreducible}, the decomposition $P = Q \times \cube[{d-r}]$ holds, with $Q$ a slice of $\cube[r]$ and the defining equation of its cutting hyperplane in $\R^r$ having strictly positive coefficients. We can thus find a combinatorially equivalent slice $Q'$ such that all simple vertices of $Q'$ (which have degree $r-1$) lie on relative interiors of edges of $\cube[r]$. These are precisely the vertices that give rise to simple vertices of $P' = Q' \times \cube[{d-r}]$, i.e., the simple vertices of $P'$ lie in relative interiors of edges of $\cube[r] \times \cube[{d-r}] = \cube[d]$. 
\end{proof}
	
We conclude this part mentioning interesting consequences of the above results. These facts are not strictly relevant for the recovery of slices from their graph, so the reader may jump to the next section. However, they could be useful for optimizing a potential implementation of the algorithm. We focus on admissible degrees of slices, when $d\geq 3$.

Let $v$ be a vertex of a slice $P=\cube[d]\cap H$ which is also a vertex of the cube. Using the notation of \cref{prop:deg_vertices}, we have $\deg_P(v)=c+ab$ with $a+b+c = d$. For fixed $c$, the product $ab$ is maximized when $a$ and $b$ are as close as possible. Hence, 
\[
\deg_P(v) \leq c+\left\lfloor \frac{(d-c)^2}{4}\right\rfloor \leq
\left\lfloor \frac{d^2}{4}\right\rfloor.
\]
Indeed, the maximum possible degree of a vertex of a slice of $\cube[d]$ is the right hand side, and it is attained when $c=0$ and 
\[
\{a,b\} = \left\{ \left\lfloor\frac d2\right\rfloor, \left\lceil\frac d2\right\rceil \right\}.
\]
In this way, we recover the main theorem in \cite{ONeil71:VerticesSlice}. If a slice $P$ contains a vertex with such maximal degree, then it must be necessarily irreducible.
Indeed, the reducibility of a slice adds restrictions on the degrees of its vertices. If $P$ is reducible of rank $r<d$, because the degree of a product is additive, it turns out that it cannot reach the maximal degree, and also that the degree of any non-simple vertex of $P$ is at least $d+r-4$. 
Finally, as a consequence, a slice with a non-simple vertex has rank $4$ if and only if it has a non-simple vertex of degree $d$. In this case, all its non-simple vertices have degree $d$.

\begin{remark}\label{rmk:slice_slab_correspondance}
    Every slice of $\cube[d]$ can be viewed affinely as a slab of a lower-dimensional cube, i.e., as the intersection of $\cube[{d-1}]$ with two halfspaces which are bounded by parallel hyperplanes. Indeed, let $P=\cube[d]\cap H$. Without loss of generality, by \labelcref{assumption:non_negative} and possibly permuting coordinates, we may assume that $u_d>0$. On $H$, the last coordinate is determined by $x_d=\frac{1}{u_d} \left(t-\sum_{i=1}^{d-1}u_ix_i\right)$. Therefore, projection onto the first $d-1$ coordinates gives an affine bijection between $P$ and
    \[
    S= \set{x\in\cube[d-1]}{t-u_d\leq \sum_{i=1}^{d-1}u_ix_i \leq t}.
    \]
    Thus $P$ is affinely equivalent to the portion of $\cube[d-1]$ lying between two parallel hyperplanes. Under this correspondence, every vertex of $P$ which is also a vertex of $\cube[d]$ projects to a cube vertex lying on one of the two boundary hyperplanes of the slab.
\end{remark}

\begin{example}[Slices as slabs]\label{ex:slice_slab}
    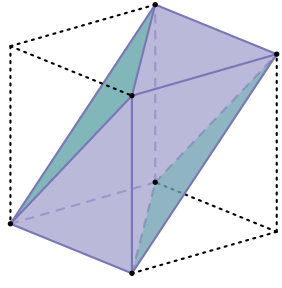
\begin{figure}[ht!]
		\centering
		\tdplotsetmaincoords{70}{50}

\begin{tikzpicture}[tdplot_main_coords, scale=2.5, line join=round, line cap=round]

\definecolor{poly_color}{named}{Periwinkle}
\definecolor{oblique_color}{named}{PineGreen}

\coordinate (A) at (0,0,0);
\coordinate (B) at (1,0,0);
\coordinate (C) at (0,1,0);
\coordinate (D) at (1,0,1);
\coordinate (E) at (0,1,1);
\coordinate (F) at (1,1,1);
\coordinate (G) at (1,1,0);
\coordinate (H) at (0,0,1);

\tikzset{facet/.style={draw=none, fill opacity=0.3}}

\draw[thick,dotted] (G)--(C);
\draw[thick,dashed,poly_color] (C)--(F);
\draw[thick,dashed,poly_color] (C)--(E);
\draw[thick,dashed,poly_color] (B)--(C);
\draw[thick,dashed,poly_color] (A)--(C);


\filldraw[facet, fill=white, fill opacity=0.5] (A)--(B)--(F)--(E)--cycle;

\filldraw[facet, fill=oblique_color, fill opacity=0.4] (B)--(C)--(F)--cycle;

\filldraw[facet, fill=poly_color] (A)--(B)--(C)--cycle;

\filldraw[facet, fill=poly_color] (A)--(B)--(D)--cycle;

\filldraw[facet, fill=poly_color] (D)--(E)--(F)--cycle;

\filldraw[facet, fill=poly_color] (A)--(C)--(E)--cycle;

\filldraw[facet, fill=poly_color] (B)--(D)--(F)--cycle;

\filldraw[facet, fill=poly_color] (C)--(E)--(F)--cycle;

\filldraw[facet, fill=oblique_color, fill opacity=0.4] (A)--(D)--(E)--cycle;



\draw[thick,poly_color] (A)--(E)--(F)--(D)--cycle;
\draw[thick,poly_color] (B)--(D);
\draw[thick,poly_color] (A)--(B);
\draw[thick,poly_color] (D)--(E);
\draw[thick,poly_color] (B)--(F);

\draw[thick,dotted] (H)--(A);
\draw[thick,dotted] (H)--(D);
\draw[thick,dotted] (H)--(E);
\draw[thick,dotted] (G)--(B);
\draw[thick,dotted] (G)--(F);

\foreach \P in {A,B,C,D,E,F}
  \fill (\P) circle (0.4pt);

\end{tikzpicture}
		\caption{The octahedral slab $S$ of $\cube[3]$ from \cref{ex:slice_slab}. }
		\label{fig:slice_slab}
	\end{figure}
	Consider the slice of $\cube[4]$ defined by the hyperplane 
	\[
	   H = \set{x \in \R^4}{x_1+x_2-x_3-x_4=0}.
	\]
	This is an octahedron. We can view it, according to \cref{rmk:slice_slab_correspondance}, as a slab of $\cube[3]$. Indeed, it corresponds to 
	\[
	S = \set{x \in C_3}{0 \leq x_1+x_2-x_3 \leq 1},
	\]
	shown in violet in \cref{fig:slice_slab}. The two parallel hyperplanes cutting out the slab are shaded in green.
\end{example}

\subsection{Hypersimplex slices}
A special family of polytopes that arise as cube slices is that of \emph{hypersimplices}.

\begin{definition}
    The \defn{$(d,k)$-hypersimplex} is the $(d-1)$-dimensional polytope
    \[
    \Delta(d,k) = \cube \cap \set{x \in \R^d}{\sum_{i=1}^d x_i = k}.
    \]
    Note that $\Delta(d,k) \cong \Delta(d,d-k)$ are combinatorially (and even affinely) equivalent.
\end{definition}

Their definition places them automatically as cube slices, but actually they cover all possible combinatorial slices with specific degree patterns. This is the content of the next results.

\begin{proposition}
    \label{prop:hypersimplices}
    Let $P=\cube\cap H$ be a slice so that $\verts(P)\subset\verts(\cube)$. Consider $v$ a vertex of $P$ and let $c$ be the number of vertices of $\cube$ adjacent to $v$ contained in $H$, and $a$ the number of vertices of $\cube$ adjacent to $v$ that lie above $H$. Then $P=\Delta(d-c,a)\times \cube[{c}]$.
\end{proposition}

\begin{proof}
    Without loss of generality, see \ref{assumption:through_origin}, assume that $H=u^\perp$ and $v=\zero$. 
    If all the vertices of $\cube$ adjacent to $\zero$ lie on one side of $H$ or in $H$ itself, then $H$ is a supporting hyperplane defining an $c$-dimensional face of~$\cube$, and $a \in \{0, d - c\}$. Since $\Delta(d-c,0) \cong\Delta(d-c,d-c)$ is a single point, we have $P = \Delta(d-c,a) \times \cube[c]=\cube[c]$.
    
    We now consider the case where some edges of $\cube$ lie above $H$, and some lie below. Let $A$ be the set of indices of neighbors of $v=\zero$ which lie above $H$, $B$ of those that lie below, and $C$ of those that are contained in $H$. For each $i \in C$, we have $e_i \in u^\perp$ and hence $u_i = 0$.
    For each $i \in A$ and $j \in B$, $H$ intersects the $2$-dimensional face $F_{i,j} = \conv(0, e_i, e_j, e_i + e_j)$ in its relative interior, separating the vertices $e_i$ and $e_j$. In particular, $H$ does not contain any edge of $F_{i,j}$. As $H$ does not intersect any edge of $\cube$ in its relative interior by hypothesis, we have $F_{i,j} \cap H = \conv(0, e_i + e_j)$. Hence, $\langle u, e_i + e_j \rangle = u_i + u_j = 0$, so $u_i = - u_j$. Ranging over all such pairs $i,j$ implies that $u \in \{0,\pm u_1\}^d$, and, up to a non-zero multiple, we can assume that 
    \[
    u_i = \begin{cases}
        1 & \text{if } e_i \text{ lies above } H, \\
        -1 & \text{if } e_i \text{ lies below } H, \\
        0 & \text{otherwise}. \
    \end{cases}
    \]
    We now define the affine map $\varphi(x) = Mx +b$ with the diagonal matrix $M \in \R^{d \times d}$ and $b\in \R^d$ as
    \[
    M_{ii} = \begin{cases}
        -1 & \text{if } e_i \text{ lies above } H, \\
        1 & \text{otherwise}, \ 
    \end{cases}
    \qquad
    \text{and}
    \qquad
    b_{i} = \begin{cases}
        1 & \text{if } e_i \text{ lies above } H, \\
        0 & \text{otherwise}. \
    \end{cases}
    \]
    Then $\varphi(\cube) = \cube$ and
    \[
    \varphi(H) = \set{x \in \R^d}{\sum_{\substack{i \text{ s.t.}\\ e_i \not\in H}} x_i = a} .
    \]
    Up to permutation of coordinates, $\varphi(P)$ equals $\Delta(d-c,a) \times \cube[{c}]$. 
\end{proof}

As a consequence, cube slices that meet the relative interior of an edge only if they contain the edge have constant vertex degree. More precisely, the vertices of $\Delta(d-c,a) \times \cube[{c}]$ have degree 
\[
a(d-c-a) + c.
\]
Hence, such a slice is either simple or it has no simple vertex. 

The following statement proves that slices of a cube with the largest possible number of non-simple vertices must arise as a Cartesian product of the octahedron $\Delta(4,2)$ (cf. \cref{ex:slice_slab}) and a cube of the correct dimension. As before, this is not strictly relevant for the general reconstruction, but may be helpful for practical purposes and optimized implementations of a recovering algorithm.

\begin{proposition}
    For $d\geq 4$, the slices of $\cube$ with the maximum number of non-simple vertices are combinatorially equivalent to the product $\Delta(4,2) \times \cube[{d-4}]$ of a $(4,2)$-hypersimplex (octahedron) and a cube.
\end{proposition}
\begin{proof}
    We proceed by induction on $d$. For $d=4$ one can explicitly compute the $30$ combinatorial types of slices, see \cite{BraMer25:6Cube}. Among these, there is a unique combinatorial type with 6 vertices, each of them being non-simple and of degree 4, and there is no other type with more non-simple vertices. By \cref{prop:hypersimplices}, this slice is an octahedron, namely $\Delta(4,2)$. 
    
    Assume now that $d>4$, that the statement is true for $\cube[{d-1}]$, and let $P$ be a slice of $\cube$ with the maximum number of non-simple vertices. If $P$ factorizes as $P=Q \times \cube[{d-r}]$, where $Q$ is a slice of $\cube[r]$, $r \leq d-1$, then by the inductive hypothesis $P = \Delta(4,2) \times \cube[{d-4}]$ and it has $6\cdot 2^{d-4}$ non-simple vertices. 
    
    Assume by contradiction that $P$ is irreducible. Without loss of generality, see \ref{assumption:non_negative}, $P=\cube\cap H$ for $H = \{\sum u_i x_i = t\}$ with $u_i>0$ for all $i=1,\ldots,d$. The non-simple vertices of $P$ must be vertices of $\cube$. Let $v, w$ be such vertices. Then $\supp (v)$ and $\supp (w)$ must be incomparable subsets of $[d]$: the supports of the non-simple vertices form what is called an antichain in $2^{[d]}$. Indeed, if $\supp (v) \subset \supp (w)$, then 
    \[
    t = \sum_{i \in \supp (v)} u_i < \sum_{i \in \supp (w)} u_i = t
    \]
    gives a contradiction. By Sperner's theorem, an antichain has at most $\binom{d}{\lfloor d/2 \rfloor}$ members, providing an upper bound on the number of non-simple vertices of $P$. 
    By a simple inductive argument with Pascal's binomial identity, one can show that 
    \[
    \binom{d}{\lfloor d/2 \rfloor} < 6 \cdot 2^{d-4},
    \]
    and the claim follows.
\end{proof}
	
\section{Determining the dimension}\label{sec: dim}
In this section, we will see that the graph of a cube slice completely determines its dimension.
We already proved that when all vertices of a section are vertices of the cube, the combinatorial type of the slice is determined. Thanks to the following result, we can actually determine if this is the case from the graph, and then obtain the dimension of the slice.

\begin{lemma}
    \label{lemma:hypersimplex_not_simple_slice}
    Let $2\leq k\leq n-2$. Then the graph of the product $\Delta(n,k)\times \cube[m]$ is not the graph of any simple polytope.
\end{lemma}
\begin{proof}
    The graph of $\Delta(n,k)\times \cube[m]$ is the Cartesian product of the Johnson graph $J(n,k)$ with the $m$-dimensional cube graph. Its vertices are therefore pairs $(I,u)$, where $I$ is a $k$-subset of $[n]$ and $u$ is a vertex of the $m$-cube. Two such pairs are adjacent if either the Johnson coordinate changes by one element and the cube coordinate stays fixed, or the cube coordinate changes by one cube edge and the Johnson coordinate stays fixed.
    
    Assume by contradiction that this is the graph of a simple polytope $P$. Fix any subset $S$ of $[n]$ of size $k-2$ and let $x,y,z,t \in [n]\setminus S$. They exist because $k\leq n-2$. Our goal is to construct two $4$-cycles, all lying in the fixed cube fiber over $u$, that cannot simultaneously be both faces of $P$. These are 
    \begin{align*}
        \mathfrak{C}_1 &= \big((S\cup x\cup y,u), (S\cup x\cup z,u), (S\cup z\cup t,u), (S\cup y\cup t,u), (S\cup x\cup y,u) \big),\\
        \mathfrak{C}_2 &= \big((S\cup x\cup y,u), (S\cup x\cup t,u), (S\cup z\cup t,u), (S\cup y\cup z,u), (S\cup x\cup y,u)).
    \end{align*}
    Both cycles are induced: the opposite vertices differ by two exchanges and the cube coordinate stays the same.
    By \cite[Proposition~1.10]{PfeiflePilaudSantos}, $\mathfrak{C}_1$ and $\mathfrak{C}_2$ must be graphs of $2$-faces of $P$, which we denote by $F_1$ and $F_2$ respectively. 
    
    The two cycles have two vertices in common: $(S\cup x\cup y,u)$ and $(S\cup z\cup t,u)$. Therefore, the intersection $F=F_1\cap F_2$ is a face of $P$, and more precisely the segment connecting those two vertices.
    However, they are not adjacent so such edge does not exist, giving a contradiction. This proves the claim.
\end{proof}

For the final dimension-recovery result, recall that since we are interested in slices up to combinatorial type, we can always assume that all their simple vertices arise as intersections of the defining hyperplane with the relative interior of cube edge, according to \cref{prop:simple_vertices_through_edges}.
\begin{proposition}\label{prop:dimension}
    The vertex-edge graph $G$ of a slice of a cube $P=\cube\cap H$ completely determines the dimension~$d$ of $\cube$. In particular, if $G$ is isomorphic to the Cartesian product $J(n,k) \square \cubegraph[m]$ of a Johnson graph with a cube graph, then $P \cong \Delta(n,k)\times \cube[m]$ and $d = n+m$.
\end{proposition}
\begin{proof}
    We examine the two possible cases.
    If $G$ is not isomorphic to the graph of the product of a hypersimplex $\Delta(n,k), 2 \leq k \leq n-2$, with a hypercube, then we know from \cref{prop:hypersimplices} that $P$ must have a vertex that is not a vertex of the cube, and thus a simple vertex by \cref{cor:simple_vertices}. Therefore, $d-1$ is the smallest vertex degree among all the vertices of $G$.
    
    Otherwise, $G$ is isomorphic to the graph of a product of a hypersimplex $\Delta(n,k), 2 \leq k \leq n-2$, with a cube. Then, by \cref{lemma:hypersimplex_not_simple_slice}, it cannot have any simple vertex. Because of \cref{prop:simple_vertices_through_edges}, then all the vertices of $P$ are cube vertices
    \[
    \verts(P)\subset \verts(\cube[d]).
    \]
    Therefore, by \cref{prop:hypersimplices} this implies that $P$ is isomorphic to the Cartesian product of the given hypersimplex $\Delta(n,k)$ with a cube.
\end{proof}

\section{Recognizing 2-dimensional faces from the graph}\label{sec:2-faces}
As a first step in order to determine the edge labeled vertex figures of our cube slices, we will need to detect from the graph which edges adjacent to a non-simple vertex form $2$-dimensional faces. A fundamental fact is the following bound on the sizes of $2$-faces, which implies that their edges must belong to an induced cycle of length three, four, or five. 

\begin{lemma}
Let $P=\cube\cap H\subset \R^d$ be a non-simple slice of the cube, and $v$ a non-simple vertex. Then, every $2$-face of $P$ containing $v$ is either a triangle, a quadrilateral, or a pentagon.
\end{lemma}
\begin{proof}
Every $2$-face of $P$ arises as a slice of a $3$-face of $\cube$ with $H$. Slices of $\cube[3]$ can be $k$-gons with $3\leq k\leq 6$, but only triangle, quadrilateral, and pentagon sections can arise when the hyperplane is constrained to go through a vertex of $\cube[3]$ \cite{BraMer25:6Cube}.
\end{proof}
	
In the following, we analyze all possible $k$-gonal faces of $P$ for $k=3,4,5$.

\subsection{Triangles}
Triangular faces are the easiest to treat, and we prove that whenever there is a $3$-cycle in the graph of a cube slice, it corresponds to an actual $2$-face.
\begin{proposition}\label{prop:triangles}
    Let $P=\cube\cap H\subset \R^d$ be a cube slice with vertex-edge graph $G$. Then every $3$-cycle in $G$ is a the graph of a $2$-face of $P$.
\end{proposition}
\begin{proof}
    Let $u,v,w$ be the vertices of the slice that form a $3$-cycle in $G$. Up to symmetry, we can assume that the support of $u$ is minimal. Namely, $u=\zero$ if it is a vertex of $\cube$, or $u$ lies on the relative interior of an edge incident to $\zero$, in which case the support of $u$ is a singleton. Thus, 
    \[
    \supp(u) \in \{\emptyset, \{i\}\} \quad \text{ for some }i \in [d]. 
    \]
    Since $[u,v]$ is an edge of $P$, it is the intersection of $H$ with a face of $\cube$ of dimension at most $2$. This $2$-face contains $u$, and thus also the minimal face of $\cube$ containing $u$. All such $2$-faces of $\cube$ are of the form $F_{i,j} = \conv(\zero,e_i,e_j,e_i + e_j)$. Therefore, the support of $v$ satisfies $\supp(v) \subseteq \{i,j\}$.
    By a similar argument, since $[u,w]$ forms an edge, $\supp(w) \subseteq \{k,l\}$ for some $k,l \in [d]$. 
    
    Let $Q$ be the face of $\cube$ in which precisely the coordinates $\supp(v)\cup \supp(w) \subseteq \{i,j,k,l\}$ are allowed to vary and all other coordinates are fixed equal to $0$. Then, $Q$ is the minimal face of $\cube$ containing all $3$ vertices $u,v,w$ of the cycle.
    Since $[v,w]$ is an edge of $P$, there exists some $2$-face $F$ of $Q$ such that $[v,w] = F\cap H$. This implies that $v$ and $w$ can differ in at most two coordinates, i.e., the symmetric difference of $\supp(v)$ and $\supp(w)$ has $|\supp(v)\symdiff \supp(w)|\leq 2$. This can only happen if $|\supp(v)\cup \supp(w)|\leq 3$, and hence we have $\dim(Q) \leq 3$.
    
    Because $u,v,w$ form a cycle, they are not collinear. Hence, $\dim(Q) = 3$ and $\dim(Q \cap H) = 2$.
    Since every intersection of $H$ with a face of $\cube$ defines a face of~$P$, the intersection $Q \cap H$ defines a $2$-dimensional face of~$P$. It contains all three vertices $u,v,w$, as well as the edges $[u,v], [v,w], [u,w]$ of $P$. Hence, $Q \cap H$ is the triangular $2$-face of $P$ with vertices $u,v,w$. The claim follows.
\end{proof}

\subsection{Quadrilaterals}
We now analyze $4$-cycles incident to at least one non-simple vertex. As seen already for the octahedron $\Delta(4,2)$, there can be induced $4$-cycles (such as those formed by any of the pairs of parallel cube edges in \cref{fig:slice_slab}) in the graph of a slice $P$ that do not define $2$-faces of $P$. Recall that for any $I\subset [d]$, $e_I$ is the vertex of $\cube$ with support equal to $I$.

\begin{proposition}\label{prop:4cyclesimple}
    Let $P=\cube\cap H$ be a non-simple cube slice with vertex-edge graph $G$. Let
    \[
    \mathfrak{C}=(v_0,v_1,v_2,v_3,v_0)
    \]
    be an induced $4$-cycle in $G$ with $v_0$ a non-simple vertex. 
    If at least one among $v_1, v_2, v_3$ is simple, then $\mathfrak{C}$ is the graph of a quadrilateral $2$-face of $P$.
\end{proposition}
\begin{proof}
    Without loss of generality, see \ref{assumption:through_origin}, assume that the non-simple vertex is the origin. 
    Thus $v_0 = \zero$ and $\mathfrak{C}=(\zero,v_1,v_2,v_3,\zero)$.
    Moreover we can assume, by \cref{prop:simple_vertices_through_edges}, that every simple vertex of $P$ is obtained as the intersection of $H$ with the relative interior of an edge of $\cube$.
    
    Suppose first that $v_1$ is simple. Since $v_1$ is adjacent to $\zero$ and is not a vertex of $\cube$, we have that $v_1$ lies in the interior of the cube-edge $E_1 = \conv(e_i,e_{ij})$ for some $i,j \in [d]$, and $|\supp(v_1)|=2$. Next, we argue that $|\supp(v_2)| = 3$.
    Since $[v_1,v_2]$ is an edge of $P$ and $v_1$ is not a vertex of $\cube$, we have that $[v_1,v_2]$ is the intersection of $H$ with the relative interior of a $2$-dimensional cube-face $F_{12}$. Note that $F_{12}$ contains $E_1$. Hence, $F_{12} = \conv(e_i, e_{ij}, e_{ik}, e_{ijk})$ for some $k \in [d]\setminus\{i,j\}$. The point $v_2$ either lies on a vertex of $F_{12}$ different from $e_i$ and $e_{ij}$, or it lies in the interior of an edge of $F_{12}$ different from $E_1$. As $\mathfrak{C}$ is induced, there is no edge of $P$ between $v_2$ and $\zero$. Hence, there is no $2$-dimensional cube-face containing both $v_2$ and $\zero$.
    This implies that $v_2$ is not contained in the edge $\conv(e_i, e_{ik})$.
    Thus, 
    \[
    \supp(v_2) = \{i,j,k\}.
    \]
    Now $v_3$ is adjacent to $\zero$, so $|\supp(v_3)|\leq 2$. Since $[v_2,v_3]$ is an edge of $P$, there exists some 2-dimensional cube face containing $v_2$ and $v_3$. This implies that $v_2$ and $v_3$ differ in at most $2$ coordinates, i.e., 
    \[
    |\supp(v_2) \symdiff \supp(v_3)| \leq 2.
    \]
    This can only happen if $\supp(v_3) \subset \supp(v_2)$.
    Thus the four vertices of $\mathfrak{C}$ are contained in the $3$-dimensional face $Q$ of $\cube$ spanned by the coordinates $i,j,k$. Since the cycle contains a simple vertex, $H$ does not contain $Q$. Moreover, the vertices of the cycle are not collinear. It follows that the $2$-dimensional face $Q \cap H$ of $P$ contains the vertices $\zero,v_1,v_2,v_3$ of $P$, as well as the edges $[\zero, v_1], [v_1,v_2], [v_2,v_3], [v_3,\zero]$ of $P$. Hence, $Q \cap H$ is the quadrilateral $2$-face of $P$ with vertices $\zero,v_1,v_2,v_3$.
    
    The case where $v_3$ is simple is symmetric to the case where $v_1$ is simple.
    It thus remains to consider the case where $v_2$ is simple and both $v_1$ and $v_3$ are non-simple. 
    Since $[\zero, v_1], [\zero, v_3]$ are edges of $P$, we have that $|\supp(v_1)| \leq 2$ and $|\supp(v_3)| \leq 2$. 
    Recall that, since $v_2$ is simple in $P$, it lies on the relative interior of a cube edge, by assumption.
    Denote by $J = \set{j \in [d]}{(v_2)_j = 1}$. Then,
    \[
    \supp(v_2)=J\cup\{i\},
    \]
    where $i \in [d]\setminus J$ is the unique coordinate for which $0<(v_2)_i<1$, so that $v_2 \in E_2 = \conv(e_J, e_{J \cup \{i\}})$.
    Since $v_2$ is not adjacent to $\zero$, we have $|J|\geq 2$. 
    On the other hand, the edges $[v_1,v_2]$ and $[v_2,v_3]$ are intersections of $H$ with $2$-dimensional cube faces. Hence, 
    \[
    |\supp(v_1) \symdiff \supp(v_2)| \leq 2 
    \quad \text{and} \quad |\supp(v_2) \symdiff \supp(v_3)| \leq 2.
    \]
    Each of these inequalities implies that $\supp(v_2)$ can have at most $2$ elements outside of $\supp(v_1)$ and of $\supp(v_3)$. Hence, $|J\cup\{i\}|\leq 4$, so $|J|\leq 3$.    
    Assume for contradiction that $|J|= 3$. Then for both $k=1,3$, the condition  $|\supp(v_k)|\leq 2$ together with $|\supp(v_k) \symdiff \supp(v_2)| \leq 2$ implies $\supp(v_k) \subset J\cup\{i\}$. Hence, $v_k$ and $v_2$ lie in a common $2$-dimension face $F_{2k}$ of $\cube$ in which the free coordinates are $B_k = (J \cup \{i\}) \setminus \supp(v_k)$. Note that $F_{2k}$ contains the edge $E_2$, and hence $i \in B_k$. This yields $i \not \in \supp(v_k)$, and hence $\supp(v_k) \subset J$. Thus, $|\supp(v_1) \symdiff \supp(v_3)| \leq 2$, so $v_1,v_3$ are contained in a common $2$-dimensional cube-face $F_{13}$. However, in this case $F_{13} \cap H$ defines an edge of $P$, thus yielding a contradiction to the assumption that $\mathfrak{C}$ is an induced cycle of the graph of $P$. Hence $|J|=2$. 
    
    Thus the four vertices of $\mathfrak{C}$ are contained in the $3$-dimensional face $Q$ of $\cube$ spanned by the coordinates $J\cup\{i\}$. Similarly to above, it follows that the $2$-dimensional face $Q \cap H$ is the quadrilateral $2$-face of $P$ with vertices $\zero,v_1,v_2,v_3$.		
\end{proof}
	
The hardest case is when the cycle consists of only non-simple vertices. If the slice is irreducible though, such a cycle can never form a $2$-face.
\begin{proposition}\label{prop:4cycles_nonsimple_irred}
    Let $P=\cube\cap H\subset \R^d$ be an irreducible slice. Let $G$ be the vertex-edge graph of $P$. Let
    \[
    \mathfrak{C}=(v_0,v_1,v_2,v_3,v_0)
    \]
    be an induced $4$-cycle in $G$ whose vertices are all non-simple. Then $\mathfrak{C}$ is not the graph of a quadrilateral $2$-face of $P$.
\end{proposition}

\begin{proof}
    Without loss of generality, see \ref{assumption:through_origin}, assume that $v_0 = \zero$, thus $\mathfrak{C}=(\zero,v_1,v_2,v_3,\zero)$ and the hyperplane is $H=u^\perp$. The slice is irreducible so, by \cref{prop:u-iff-irreducible}, $u_i\neq 0$ for all $i \in [d]$ and the hyperplane $H$ contains no edge of the cube.
    Suppose, for contradiction, that $\mathfrak{C}$ is the graph of a quadrilateral $2$-face $F$ of $P$. Then there is a $3$-dimensional face $Q$ of $\cube$ such that $F=Q\cap H$.
    After relabeling coordinates, we may assume that $Q$ is spanned by $e_1,e_2,e_3$. Since $v_1$ and $v_3$ are adjacent to the origin in $P$ and are cube vertices, their supports must have size two (note that $H$ contains no edge, so they cannot have size one). Thus,
    \[
    \supp(v_1),\supp(v_3) \subseteq \{1,2,3\} \quad \text{and} \quad |\supp(v_1)| = |\supp(v_3)| = 2.
    \]
    Therefore, these two supports have a common coordinate, hence $v_1$ and $v_3$ are contained in a common $2$-dimensional cube-face $F$. Since $H$ does not contain any edge of $\cube$, it follows that $F\cap H$ is a $1$-dimensional face of $P$ with vertices $v_1$ and $v_3$. 
    This contradicts the assumption that $\mathfrak{C}$ is an induced $4$-cycle. Hence $\mathfrak{C}$ is not the graph of a quadrilateral $2$-face of $P$.
\end{proof}

This result requires irreducibility. It gives us a complete understanding of which $4$-cycles form $2$-faces of a cube slice, provided that we know its rank. In \cref{sec:decomposition} we will show how to recover the rank. Note that by \cref{thm:cartesian_product}, graphs admit a unique decomposition as a Cartesian product of prime graphs (up to isomorphism of the factors). 
Therefore, if $P=Q\times \cube[{d-r}]$ is a reducible slice of rank~$r$, we can recover the vertex-edge graph of $Q$ from the vertex-edge graph $G$ of $P$ by just removing $d-r$ single-edge factors from this cartesian product. Because of this decomposition, we can identify $2^{d-r}$ copies of the vertex-edge graph of $Q$ inside of $G$.
    
In the following more detailed statement, we classify all possible occurrences of non-simple $4$-cycles in (not necessarily irreducible) cube slices, and determine when they define $2$-faces. This will help us later in \cref{sec:decomposition} in finding the correct rank of a slice.
\begin{proposition}\label{prop:reducible_quadrilaterals}
    Let $P=\cube\cap H$ be a reducible slice of rank $r$ with decomposition $P\cong Q\times \cube[{d-r}]$. Let $G$ be the vertex-edge graph of $P$ and let
    \[
    \mathfrak{C}=(v_0,v_1,v_2,v_3,v_0)
    \]
    be an induced $4$-cycle in $G$ whose vertices are all non-simple. Then $\mathfrak{C}$ necessarily satisfies one of the following:
    \begin{itemize}
        \item[(i)] If $\mathfrak{C}$ is contained in a subgraph of $G$ that corresponds to a copy of the vertex-edge graph of $Q$, then $\mathfrak{C}$ is not the graph of a quadrilateral $2$-face of $P$;
        \item[(ii)] If $d-r\geq 2$ and $\mathfrak{C}$ is contained in a subgraph of $G$ that corresponds to a copy of $\cube[{d-r}]$, then $\mathfrak{C}$ is the graph of a quadrilateral $2$-face of $P$;
        \item[(iii)] If the vertices $\{v_0,v_1\}$ and $\{v_2,v_3\}$ are contained in two different copies of $Q$, then $\mathfrak{C}$ is the graph of a quadrilateral $2$-face of $P$.
    \end{itemize}
\end{proposition}
	
\begin{proof}
    First, we observe that the cases listed in the statement cover all possible distributions of the vertices of $\mathfrak{C}$ across copies of $Q$. 
    If all of the vertices lie in one copy of $Q$, then (i) applies.
    If they each lie in a different copy of $Q$, then (ii) applies.
    If the vertices lie in three different copies of $Q$, then $\mathfrak{C}$ cannot be an induced cycle.
    Indeed, if $v_0,v_1$ lie in one copy and that $v_2$ and $v_3$ lie in two further copies, we write $v_0=(q_0,c_0)$ and $v_1=(q_1,c_0)$ for the vertices $q_i$ of $Q$ and $c_0$ of $\cube[{d-r}]$. Since the remaining three edges of $\mathfrak C$ join distinct copies of $Q$, we have that
    \[
    v_2=(q_1,c_2), \qquad v_3=(q_1,c_3)=(q_0,c_3),
    \]
    for some $c_2, c_3 \in \verts(\cube[{d-r}])$.
    Hence $q_0=q_1$, and therefore $v_0=v_1$, a contradiction to the assumption that $\mathfrak{C}$ is an induced $4$-cycle.
    On the other hand, if $v_0,v_2$ lie in the same copy, then $v_0 = (q_0,c_0)$, $v_1 = (q_0,c_1)$ and hence $v_2 = (q_0,c_0) = v_0$, and we derive a similar contradiction. Hence, $v_0,v_1,v_2,v_3$ cannot lie in three different copies of $Q$.
    Finally, condition (iii) is one configuration in which the vertices of $\mathfrak{C}$ lie in $2$ copies of $Q$. The only other partition (up to relabeling) gives $\{v_0,v_2\}$ and $\{v_1,v_3\}$ contained in two different copies of $Q$, which is impossible  since this would imply $v_0 = v_2, v_1 = v_3$ similarly as before. 
    
    Now we proceed in deciding which of these admissible cycles define faces of $P$. Without loss of generality, see \ref{assumption:through_origin}, $v_0=\zero$. Then,
    \begin{itemize}
        \item [(i)] In this case, the claim follows from \cref{prop:4cycles_nonsimple_irred}.
        \item[(ii)] If $\mathfrak{C}$ is contained in one copy of $\cube[d-r]$, then it corresponds to a $2$-face of $\cube[d-r]$ and hence of $P$.
        \item[(iii)] In this case, $[\zero, v_3]$ and $[v_1,v_2]$ correspond to edges of $\cube[d-r]$,
        so $v_0 = \zero = (q_0,c_0)$, $v_1 = (q_1,c_0), v_2 = (q_1,c_1), v_3 = (q_0,c_1)$. Thus, $\mathfrak{C}=[q_0,q_1]\square[c_0,c_1]$. Since $\conv(q_0,q_1)$ is an edge of $Q$ and $\conv(c_0,c_1)$ is an edge of $\cube[d-r]$, we have that $\conv(q_0,q_1)\times \conv(c_0,c_1)$ is a $2$-face of $Q \times \cube[d-r]$. Hence $\mathfrak{C}$ is the graph of a $2$-face of $P$.
        \qedhere
    \end{itemize}
\end{proof}
	
\subsection{Pentagons}
Now we move on to $5$-cycles inside the vertex-edge graph of a non-simple slice of a cube. A necessary condition for an induced $5$-cycle to be a $2$-dimensional face is that it is the intersection of a hyperplane with a $3$-dimensional cube. This forces exactly one of the vertices of the cycle to be a vertex of the $3$-dimensional cube. The others are intersections with edges (hence, necessarily simple; see \Cref{fig:pentagonal_slice}). We make use of the following convexity notion, initially introduced in \cite{ASdSS13:P3*convex}.
    
\begin{definition}
    Let $G$ be a graph and let $\mathfrak{C}=(v_0,v_1,\ldots,v_{k-1},v_0)$ be an induced $k$-cycle. We say that $\mathfrak{C}$ is \defn{$P_3^*$-convex} if the neighbors satisfy
    \[
    N_G(v_{i})\cap N_G(v_{j})\subseteq \verts(\mathfrak{C})
    \]
    for every $i,j=0,\ldots,k-1$, $i \neq j$ such that $v_i \not\in N_G(v_{j})$.
\end{definition}
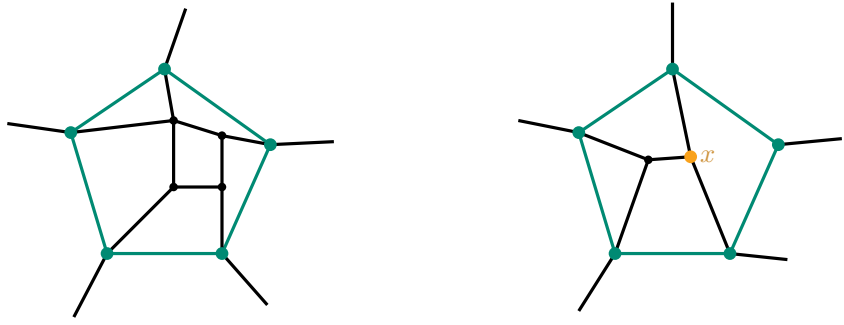
\begin{figure}[ht]
    \centering
    \begin{tikzpicture}[scale=0.8]

\definecolor{join_color}{named}{YellowOrange}
\definecolor{product_color}{named}{PineGreen}

\tikzset{
    cycle_edge/.style={product_color, very thick},
    extra_edge/.style={black, very thick},
    bad_edge/.style={join_color, very thick},
}

\begin{scope}[xshift=-4.2cm]

    \coordinate (v1) at (0,2.0);
    \coordinate (v2) at (-1.55,0.95);
    \coordinate (v3) at (-0.95,-1.05);
    \coordinate (v4) at (0.95,-1.05);
    \coordinate (v5) at (1.75,0.75);

    \coordinate (u1) at (0.15,1.15);
    \coordinate (u2) at (0.15,0.05);
    \coordinate (u3) at (0.95,0.05);
    \coordinate (u4) at (0.95,0.90);

    \draw[cycle_edge]
        (v1)--(v2)--(v3)--(v4)--(v5)--cycle;

    \draw[extra_edge] (v1)--(u1);
    \draw[extra_edge] (v2)--(u1);
    \draw[extra_edge] (u1)--(u2)--(u3)--(u4)--(u1);
    \draw[extra_edge] (u2)--(v3);
    \draw[extra_edge] (u3)--(v4);
    \draw[extra_edge] (u4)--(v5);

    \draw[extra_edge] (v1) -- +(0.35,1.0);
    \draw[extra_edge] (v2) -- +(-1.05,0.15);
    \draw[extra_edge] (v3) -- +(-0.55,-1.05);
    \draw[extra_edge] (v4) -- +(0.75,-0.85);
    \draw[extra_edge] (v5) -- +(1.05,0.05);

    \foreach \p in {v1,v2,v3,v4,v5}{
        \fill[product_color] (\p) circle (3pt);
    }
    \foreach \p in {u1,u2,u3,u4}{
        \fill[black] (\p) circle (2pt);
    }
\end{scope}

\begin{scope}[xshift=4.2cm]

    \coordinate (w1) at (0,2.0);
    \coordinate (w2) at (-1.55,0.95);
    \coordinate (w3) at (-0.95,-1.05);
    \coordinate (w4) at (0.95,-1.05);
    \coordinate (w5) at (1.75,0.75);

    \coordinate (x) at (0.3,0.55);

    \coordinate (y) at (-0.4,0.5);

    \draw[cycle_edge]
        (w1)--(w2)--(w3)--(w4)--(w5)--cycle;

    \draw[extra_edge] (x)--(w1);
    \draw[extra_edge] (x)--(w4);

    \draw[extra_edge] (y)--(w3);
    \draw[extra_edge] (y)--(x);
    \draw[extra_edge] (y)--(w2);

    \draw[extra_edge] (w1) -- +(0.00,1.10);
    \draw[extra_edge] (w2) -- +(-1.00,0.20);
    \draw[extra_edge] (w3) -- +(-0.60,-0.95);
    \draw[extra_edge] (w4) -- +(0.95,-0.10);
    \draw[extra_edge] (w5) -- +(1.05,0.10);

    \foreach \p in {w1,w2,w3,w4,w5}{
        \fill[product_color] (\p) circle (3pt);
    }
    \fill[join_color] (x) circle (3pt);
    \fill[black] (y) circle (2pt);

    \node[right] at (x) {\textcolor{join_color!80!black}{$x$}};
\end{scope}
\end{tikzpicture}
    \caption{Graph with a $P_3^*$-convex $5$-cycle (left) and graph with a non-$P_3^*$-convex $5$-cycle (right) $5$-cycle (the $5$-cycles are displayed in green). The orange vertex $x$ on the right makes the cycle non-$P_3^*$-convex.}
    \label{fig:P3*convex_cycle}
\end{figure}
In other words, all common neighbors of two non-neighboring vertices of the cycle are vertices of the cycle. This imposes no constraint on $3$-cycles. For $5$-cycles, it says that 
\[
N_G(v_{i-1})\cap N_G(v_{i+1}) = \{v_i\}
\]
for $i$ modulo $5$. \cref{fig:P3*convex_cycle} shows an example (left) and a non-example (right) of $P_3^*$-convex $5$-cycles.
Not being $P_3^*$-convex is the only obstruction for $5$-cycles to be faces.

\begin{proposition}\label{prop:5cycles}
    Let $P=\cube\cap H\subset \R^d$ be a slice. Let $G$ be the vertex-edge graph of $P$. Let
    \[
    \mathfrak{C}=(v_0,v_1,v_2,v_3,v_4,v_0)
    \]
    be an induced $5$-cycle in $G$ containing exactly one non-simple vertex. Then
    $$ \mathfrak{C} \mbox{ is the graph of a pentagonal } 2\mbox{-face of } P \Longleftrightarrow \mathfrak{C} \mbox{ is $P_3^*$-convex.}$$
\end{proposition}
	
\begin{proof}
    We can use \ref{assumption:through_origin} to assume that $H$ contains the origin, and that it is the unique non-simple vertex of $\mathfrak{C}$. Thus, $\mathfrak{C}=(\zero,v_1,v_2,v_3,v_4,\zero)$. 
    If $v$ is a simple vertex of $P$, we denote by $E(v)$ the edge of $\cube$ whose relative interior contains $v$.\\
    
    $(\Rightarrow)$ We first prove that a facial pentagon is $P_3^*$-convex. Assume that $\mathfrak{C}$ is the graph of a $2$-face $F$ of $P$. Since $F$ is a pentagon, it cannot be a face of the cube itself. Hence, there is a $3$-face $Q$ of $\cube$ such that
    \[
    F=Q\cap H.
    \]
    Since $\zero\in F$, after relabeling coordinates we may assume that $Q$ is spanned by $e_1,e_2,e_3$, and that the restriction of $H$ to $Q$ has equation
    \[
    -u_1 x_1-u_2 x_2+u_3 x_3=0
    \]
    with $u_1,u_2,u_3>0$. 
    From this sign pattern, we obtain that the two edges $[\zero, v_1], [\zero,v_4]$ of $\mathfrak{C}$ arise as the intersection of $H$ with the $2$-dimensional cube-faces $F_{13}$ and $F_{23}$, respectively, where $F_{ij} = \conv(\zero,e_i,e_j,e_{ij})$. In order for $F$ to be a pentagon, the hyperplane $H$ must intersect the relative interior of the cube-edge $\conv(e_i, e_{i3})$ for $i=1,2$ (see \Cref{fig:pentagonal_slice}). Since $u_i > 0$, $H$ contains $e_i$ on its negative side. In order to have $e_{i3}$ on the positive side, we must have $-u_i + u_3 > 0$, and hence $u_i < u_3$. Furthermore, a pentagonal slice requires that $H$ separates $e_{13}$ and $e_{23}$ (which lie on the positive side) from $e_{123}$, so $-u_1 - u_2 + u_3 < 0$. Thus,
    \[
    \max\{u_1,u_2\}<u_3<u_1+u_2.
    \]
    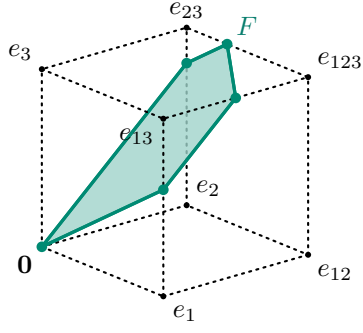
\begin{figure}[ht]
        \centering
        \tdplotsetmaincoords{70}{50}

\begin{tikzpicture}[
    tdplot_main_coords,
    scale=2.5,
    line join=round,
    line cap=round
]

\definecolor{slice_color}{named}{PineGreen}

\pgfmathsetmacro{\uone}{0.6}
\pgfmathsetmacro{\utwo}{0.8}
\pgfmathsetmacro{\uthree}{1.0}

\coordinate (A) at (0,0,0);
\coordinate (B) at (1,0,0);
\coordinate (C) at (0,1,0);
\coordinate (G) at (1,1,0);

\coordinate (H) at (0,0,1);
\coordinate (D) at (1,0,1);
\coordinate (E) at (0,1,1);
\coordinate (F) at (1,1,1);

%
\coordinate (P0) at (0,0,0);

\coordinate (P1) at
    (1,0,{\uone/\uthree});

\coordinate (P2) at
    (1,{(\uthree-\uone)/\utwo},1);

\coordinate (P3) at
    ({(\uthree-\utwo)/\uone},1,1);

\coordinate (P4) at
    (0,1,{\utwo/\uthree});

\draw[thick,dotted]
    (A)--(B)
    (A)--(C)
    (C)--(E)
    (A)--(H);

\filldraw[
    fill=white,
    fill opacity=0.7,
    draw=slice_color,
    very thick
]
    (P0)--(P1)--(P2)--(P3)--(P4)--cycle;

\filldraw[
    fill=slice_color,
    fill opacity=0.3,
    draw=slice_color,
    very thick
]
    (P0)--(P1)--(P2)--(P3)--(P4)--cycle;

\draw[thick,dotted]
    (B)--(G)--(C)
    (B)--(D)
    (H)--(D)--(F)--(E)--cycle
    (G)--(F);

\foreach \P in {A,B,C,D,E,F,G,H}{
    \fill (\P) circle (0.45pt);
}

\foreach \P in {P0,P1,P2,P3,P4}{
    \fill[slice_color] (\P) circle (0.8pt);
}

\node[above right] at (P3) {{\color{slice_color!80!black} $F$}};

\node[below left]  at (A) {$\zero$};
\node[below right] at (B) {$e_1$};
\node[above right] at (C) {$e_2$};
\node[above left]  at (H) {$e_3$};

\node[below right] at (G) {$e_{12}$};
\node[below left]  at (D) {$e_{13}$};
\node[above]       at (E) {$e_{23}$};
\node[above right] at (F) {$e_{123}$};

\end{tikzpicture}
        \caption{Geometry of a pentagonal slice $F$ of a $3$-dimensional cube.}
        \label{fig:pentagonal_slice}
    \end{figure}

    The vertices of $F$ are then, in cyclic order,
    \[		
    \zero,\quad
    e_1+\frac{u_1}{u_3}e_3,\quad
    e_1+\frac{u_3-u_1}{u_2}e_2+e_3,\quad
    \frac{u_3-u_2}{u_1}e_1+e_2+e_3,\quad
    e_2+\frac{u_2}{u_3}e_3 .
    \]
    We check that no pair of vertices at distance two in this pentagon has an extra common neighbor in $G$. Consider the vertices
    \[
    \zero
    \qquad\text{and}\qquad
    e_1+\frac{u_3-u_1}{u_2}e_2+e_3 .
    \]
    A common neighbor of them must be adjacent to the origin, hence it must lie on a cube edge contained in a square through the origin. It must also be adjacent to the second vertex, whose supporting cube edge has equations $x_1=x_3=1$ and $x_2$ free. 
    Hence, a common neighbor lies on an edge with $x_1 = 1$ and $x_3$ free, or $x_3 = 1$ and $x_1$ free. In the first case, we obtain the vertex
    \[
    e_1+\frac{u_1}{u_3}e_3,
    \]
    which indeed belongs to $\mathfrak{C}$. In the second case, the intersection of $H$ with the line $e_3 + \R e_1$ satisfies $x_1 = \frac{u_3}{u_1} > 1$, so the intersection point is not contained in the cube.
    Thus,
    \[
    N_G(\zero)\cap N_G\left(e_1+\frac{u_3-u_1}{u_2}e_2+e_3\right)
    =
    \left\{e_1+\frac{u_1}{u_3}e_3\right\}.
    \]
    The same argument, with indices $1$ and $2$ interchanged, gives the analogous statement for the other length-two path from $\zero$.
    
    For the remaining length-two paths, we use the same square-face check. The only cube edge which can be adjacent both to the edges supporting
    \[
    e_1+\frac{u_1}{u_3}e_3
    \qquad\text{and}\qquad
    \frac{u_3-u_2}{u_1}e_1+e_2+e_3
    \]
    is the edge with $x_1=x_3=1$ and $x_2$ free, giving the middle vertex $e_1+\frac{u_3-u_1}{u_2}e_2+e_3$. 
    Similarly, the only common neighbor of
    \[
    e_1+\frac{u_3-u_1}{u_2}e_2+e_3
    \qquad\text{and}\qquad
    e_2+\frac{u_2}{u_3}e_3
    \]
    is $\frac{u_3-u_2}{u_1}e_1+e_2+e_3$. 
    Finally, the only possible extra common neighbor of
    \[
    e_1+\frac{u_1}{u_3}e_3
    \qquad\text{and}\qquad
    e_2+\frac{u_2}{u_3}e_3
    \]
    would lie on the edge with $x_1=x_2=1$ and $x_3$ free. But this would require $x_3=\frac{u_1+u_2}{u_3}>1$, which is outside the cube. Hence, their only common neighbor is the origin. Therefore, $\mathfrak{C}$ is $P_3^*$-convex.\\
		
    $(\Leftarrow)$ Conversely, assume that $\mathfrak{C}$ is not the graph of a $2$-face. We show that $\mathfrak{C}$ is not $P_3^*$-convex. Since $[\zero, v_1]$ and $[\zero,v_4]$ are edges of $P$, the vertices $v_1$ and $v_4$ lie on cube edges contained in squares through the origin. In particular, their supports have size two. Since $\mathfrak{C}$ is induced, neither $v_2$ nor $v_3$ is adjacent to $\zero$, so their supports have size at least three. As $[v_1,v_2]$ and $[v_3,v_4]$ are edges of $P$, their supports have size three.
    
    If all vertices of $\mathfrak{C}$ were contained in a common $3$-face of $\cube$, then the intersection of this $3$-face with $H$ would be a polygonal $2$-face of $P$ containing the induced $5$-cycle $\mathfrak{C}$. A plane section of a $3$-cube through one of its vertices has at most five vertices, so this polygon would be exactly the pentagon $\mathfrak{C}$, contradicting the assumption that $\mathfrak{C}$ is not facial. Hence, the vertices of $\mathfrak{C}$ are not contained in a common $3$-face of the cube.
    
    On the other hand, every edge of $\mathfrak{C}$ is the section with $H$ of a square face of $\cube$. Therefore, the vertices of $\mathfrak{C}$ are contained in a $4$-face of the cube. After relabeling coordinates and possibly reversing the order of the cycle, we may write
    \[
    v_1=e_2+\lambda e_3,\qquad
    v_2=e_1+e_2+\mu e_3, \qquad
    v_3=e_1+e_2+\nu e_4,\qquad
    v_4=e_1+\eta e_4,
    \]
    with $0<\lambda,\mu,\nu,\eta<1$.
    Indeed, there is no restriction on $v_1$ except $\supp(v_1) = 2$, so we may assume $v_1 = e_2 + \lambda e_3$. From $\supp(v_2) = 3$ we obtain that $\supp(v_1) \subset \supp(v_2)$, and by a symmetric argument we obtain $\supp(v_4) \subset \supp(v_3)$. Since the cycle is not contained in a $3$-face of $\cube$, we have $|\supp(v_1)\cup \supp(v_4)|\geq 4$, and since $\supp(v_1) = \{2,3\}, |\supp(v_4)| = 2$, we may choose $\supp(v_4) = \{1,4\}$ and $v_4 = e_1 + \eta e_4$. From the containment of supports, we may thus assume $\supp(v_2) = \{1,2,3\}$ and $\supp(v_3) = \{1,2,4\}$.
    Now, since $[v_1,v_2]$ form an edge, we have $v_2 \in \{e_1 + e_2 + \mu e_3, \mu' e_1 + e_2 + e_3\}$. Similarly, from $[v_3,v_4]$ forming an edge we obtain $v_3 \in \{e_1 + e_2 + \nu e_4,  e_1 + \nu' e_2 + e_4\}$. However, $[v_2,v_3]$ leaves a unique choice for $v_2$ and $v_3$, as above.
    
    Thus, the cycle enters the $4$-face through the square spanned by the indices $2,3$, leaves through the square spanned by the indices $1,4$, and the middle edge changes from the $e_3$-direction to the $e_4$-direction.
    Restrict the equation of $H$ to this $4$-face. Since $\zero\in H$, and since the four displayed vertices lie in $H$, the equation has the form
    \[
    -u_1 x_1-u_2x_2+u_3x_3+u_4x_4=0
    \]
    with $u_1,u_2,u_3,u_4>0$. 
    From $v_1,v_2,v_3,v_4\in H$ we get
    \[
    \lambda=\frac{u_2}{u_3},\qquad
    \mu=\frac{u_1+u_2}{u_3},\qquad
    \nu=\frac{u_1+u_2}{u_4},\qquad
    \eta=\frac{u_1}{u_4}.
    \]
    Since $v_2$ and $v_3$ lie in the cube, it follows that $u_1+u_2<u_3$ and $u_1+u_2<u_4$.
    Consider the point
    \[
    w=e_1+\frac{u_1}{u_3}e_3.
    \]
    Since $u_1<u_3$, this point lies in the cube. Moreover $w\in H$, because $-u_1+u_3\frac{u_1}{u_3}=0$. 
    Thus, $w$ is a vertex of $P$.
    The point $w$ is adjacent to $\zero$: both lie in the square spanned by $e_1,e_3$, and $H$ cuts this square in the segment joining them. Also, $w$ is adjacent to $v_2$: both lie in the square defined by $x_1=1$ and $x_4=0$, with $x_2,x_3$ free, and on this square the section is the segment
    \[
    \left[ e_1+\frac{u_1}{u_3}e_3, \;
    e_1+e_2+\frac{u_1+u_2}{u_3}e_3 \right].
    \]
    Hence $w$ is a common neighbor of $\zero$ and $v_2$, as well as $v_1$, and $w\neq v_1$, so
    \[
    N_G(\zero)\cap N_G(v_2)\neq \{v_1\}.
    \]
    Hence, $\mathfrak{C}$ is not $P_3^*$-convex and the claim follows.
\end{proof}

\subsection{Two-dimensional faces of cube slices}
We summarize the results from \Cref{sec:2-faces} into a complete characterization of the cycles supporting $2$-dimensional faces of a slice of $\cube[d]$ that are incident to a non-simple vertex. Because of the characterization of $4$-cycles, this requires knowledge of the rank of $P$. 

\begin{proposition}\label{prop:2-faces_red}
	Let $P=\cube\cap H$ be a slice of rank $r\leq d$, $P\cong Q\times \cube[{d-r}]$. Let $G$ be the vertex-edge graph of $P$ and let $v$ be a non-simple vertex of $P$. Then the $2$-dimensional faces of $P$ incident to $v$ are precisely those supported by 
\begin{itemize}
    \item $3$-cycles of $G$ containing $v$ (cf. \cref{prop:triangles}),
    \item induced $4$-cycles of $G$ containing $v$ and at least one simple vertex (cf. \cref{prop:4cyclesimple,prop:4cycles_nonsimple_irred}),
    \item induced $4$-cycles $\mathfrak{C}=(v,v_1,v_2,v_3,v)$ of $G$ containing $v$ and whose vertices are all non-simple, such that either $\mathfrak{C}$ is contained in a subgraph of $G$ that corresponds to a copy of $\cube[{d-r}]$, or $\{v_0,v_1\}$ and $\{v_2,v_3\}$ are contained in two different copies of $Q$ (cf. \cref{prop:reducible_quadrilaterals}),
    \item $P_3^*$-convex induced $5$-cycles of $G$ containing $v$ and $4$ simple vertices (cf. \cref{prop:5cycles}).
\end{itemize}
\end{proposition}

Note that when the slice $P$ is known to be irreducible, the third type of cycles/faces in \cref{prop:2-faces_red} does not appear. In that case, there are exactly one type of triangular faces, one type of quadrilateral faces, and one type of pentagonal faces in $P$.

\section{Reconstruction theorem for cube slices}\label{sec: rec thm}
The goal of this section is to combine the results obtained in the previous sections in order to prove the main theorem. We start by recovering the rank of a given slice from its graph.

\subsection{Unique decomposition for slices}\label{sec:decomposition}
Let $P$ be a non-simple cube slice of rank $d-c$, that is 
\[
    P\cong Q\times\cube[c],
\]
where $Q$ is an irreducible slice of $\cube[{d-c}]$. Equivalently, by \Cref{prop:u-iff-irreducible}, the defining hyperplane $H$ of $P$ is parallel to exactly $c$ coordinate vectors. In this case, $G(P)\cong G(Q)\square \cubegraph[c]$.
Our aim is to determine $c$ (or equivalently, the rank $d-c$ of $P$), from its graph of $G(P)$.

Of course, if $P \cong Q \times \cube[c]$ then there exists a decomposition of $G(P)$ as a Cartesian product $G(Q)\square \cubegraph[c]$, but is $c$ the largest integer for which such a decomposition of the graph exists? This is unfortunately not true:
the graph $G(P)$ may have accidental Cartesian cube factors which do not come from geometric cube factors of $P$ (see \cref{ex:graph_fake_segment} below).
The following lemma shows, however, that such accidental cube factors can overestimate the true value of $c$ by at most one.
	
\begin{lemma}\label{lemma:c_or_c+1}
    Let $P=\cube\cap H$ be a non-simple slice and assume that $P$ has at least one simple vertex. Let $d-c$ be the rank of $P$. If $G(P)\cong \Gamma\square\cubegraph[t]$, then $t\leq c+1$.
\end{lemma}
\begin{proof}
    Without loss of generality, see \ref{assumption:through_origin}, $H = u^\perp$ and the vertex of $P$ at $\zero$ is non-simple.
    Write
    \[
    A=\{i\in [d]\mid u_i>0\},\qquad
    B=\{i\in [d]\mid u_i<0\},\qquad
    C=\{i\in [d]\mid u_i=0\}.
    \]
    Since the rank of $P$ is $d-c$, we have $|C|=c$. 
    Moreover, since the origin is non-simple, we have $|A|,|B|\geq 2$.
    The edges of $P$ incident to $\zero$ are of two kinds. First, for each $i\in C$ there is the cube edge $[0,e_i]\subset H$.
    Second, for each pair $(a,b)\in A\times B$, the square face of $\cube$ spanned by the vectors $e_a$ and $e_b$ meets $H$ in an edge of $P$ through $\zero$. Let $y_{ab}$ be the other endpoint of this edge.
    Thus, the edges of $P$ incident to $\zero$ are precisely
    \[
    [0,e_i]\quad (i\in C),
    \qquad
    [0,y_{ab}]\quad (a\in A,\ b\in B).
    \]
    Suppose that
    \[
    G(P)\cong \Gamma\square \cubegraph[t].
    \]
	At the vertex $\zero$, there are exactly $t$ incident edges corresponding to the $\cubegraph[t]$-factor. Denote them by
    \[
    c_1,\ldots,c_t.
    \]
	We claim that among these $t$ edges, at most one is not a true cube edge $[0,e_i]$ with $i\in C$. For this, we use two elementary properties of an edge in a Cartesian cube factor. If an edge belongs to a $\cubegraph[t]$-direction, then:
    \begin{enumerate}
        \item[(i)] its two endpoints have the same degree in $G(P)$;
        \item[(ii)] it is contained in no $3$-cycle of $G(P)$.
    \end{enumerate}
	Consider one of the edges $c_k$, $1\leq k\leq t$. If $c_k$ is not one of the true cube edges $[0,e_i]$, then it must be some edge $[0,y_{ab}]$ with $a\in A$ and $b\in B$. Write
    \[
    u_a=\alpha>0,
    \qquad
    -u_b=\beta>0.
    \]
	If $\alpha\neq \beta$, then $y_{ab}$ is not a vertex of $\cube$; it is a transverse intersection point on an edge of $\cube$ with $H$. Hence, this endpoint is a simple vertex of $P$, of degree $d-1$.
	The vertex $\zero$ is non-simple, so its degree is larger than $d-1$. Thus, the two endpoints of $[0,y_{ab}]$ have different degrees, contradicting property $(i)$. Therefore every such fake cube-factor edge must satisfy $u_a=-u_b=\lambda$ for some $\lambda >0$. This implies that $y_{ab}=e_a+e_b$.
    We now show that $u_a = \lambda$ attains the unique maximum among all entries of $u$, i.e., that $\lambda > u_{a'}$ for all $a' \in A$. Thus,
	suppose for contradiction that there exists $a'\in A\setminus \{a\}$, with $u_{a'}\geq \lambda$. 
	Then we can solve the hyperplane equation to see that 
    \[
    y_{a'b}=\frac{\lambda}{u_{a'}} e_{a'}+e_b.
    \]
	Both $y_{ab}$ and $y_{a'b}$ belong to a common $2$-face of $\cube$, with $x_b=1$ and $x_i=0$ for $i\notin \{a,a',b\}$. The intersection of $H$ with the $3$-face of $\cube$ spanned by the coordinates $a,a',b$ is the triangle with vertices $0$, $y_{ab}$ and $y_{a'b}$. This triangle in $G(P)$ contradicts property~$(ii)$.
    The same argument, interchanging the roles of $A$ and $B$, shows that $-u_b=\lambda$ is also the unique maximum among the absolute values of the negative entries. Thus
    \begin{align*}
    u_{a'}&<\lambda
    &&
    \text{for every } a'\in A,\ a'\neq a,\\
    -u_{b'}&<\lambda
    &&
    \text{for every } b'\in B,\ b'\neq b.
    \end{align*}
	Consequently, any edge $y_{ab}$ which appears as a $\cubegraph[t]$-factor edge must correspond to a strictly dominating coefficient $\lambda$.
    Hence, there is at most one such pair $(a,b)$. This proves that $t\leq c+1$.
\end{proof}

As the proof shows, we obtain the following characterization of the slices of $\cube[d]$ with one fake cube-factor in the graph.
\begin{corollary}\label{cor:shape_dominant_slice}
    Let $P = \cube\cap H$ be a slice of rank $r$ such that $G(P)\cong \Gamma \square \cubegraph[d-r+1]$. Then, up to the symmetries of the cube, the defining equation of $H$ has the form
    \[
    \lambda x_1 - \lambda x_2 + \sum_{a\in A\setminus \{1\}} u_a x_a - \sum_{b\in B\setminus \{2\}} u_b x_b = 0,
    \]
    where $u_a>0$ for all $a\in A\setminus \{1\}$, $u_b>0$ for all $b\in B\setminus \{2\}$, $\lambda > |u_i|$ for all $i \in (A\cup B) \setminus \{1,2\}$, and $|A|+|B| = r$. Moreover, the segment $[\zero, e_{12}]$ is the graph-theoretic but non-geometric extra cube factor.
\end{corollary}
    	
The next example shows that, indeed, such non-geometric but graph-theoretic cube factors may arise.
\begin{example}\label{ex:graph_fake_segment}
    Consider the slice $P=\cube[5]\cap H$, where
    \[
    H=\{x\in\mathbb R^5\mid 10x_1-10x_2+x_3-x_4-x_5=0\}.
    \]
    Since no coefficient of the defining equation is zero, the slice has no geometric cube factor; in other words $P$ is irreducible. In $P$ there are six non-simple vertices of degree $6$ and ten simple vertices. See \Cref{fig:graph_fake_segment}.
    \begin{figure}[ht]
        \centering
        \begin{tikzpicture}[scale=0.8, line join=round, line cap=round]

\definecolor{left_color}{named}{Periwinkle}
\definecolor{right_color}{named}{PineGreen}

\tikzset{
  factor edge/.style={draw=black!42, line width=0.35pt},
  product edge/.style={draw=YellowOrange, line width=0.8pt},
  left vertex/.style={circle, fill=left_color, inner sep=1.5pt},
  right vertex/.style={circle, fill=right_color, inner sep=1.5pt}
}


\coordinate (v1)  at (0.000,2.000);
\coordinate (v2)  at (1.250,3.650);
\coordinate (v3)  at (1.150,0.850);
\coordinate (v7)  at (2.450,2.150);
\coordinate (v11) at (3.250,4.550);
\coordinate (v15) at (3.250,0.000);
\coordinate (v16) at (4.900,1.050);
\coordinate (v14) at (4.950,3.450);


\begin{scope}[xshift=7.0cm, yshift=1.0cm]
\coordinate (v10) at (0.000,2.000);
\coordinate (v5)  at (1.250,3.650);
\coordinate (v8)  at (1.150,0.850);
\coordinate (v4)  at (2.450,2.150);
\coordinate (v6)  at (3.250,4.550);
\coordinate (v12) at (3.250,0.000);
\coordinate (v9)  at (4.900,1.050);
\coordinate (v13) at (4.950,3.450);
\end{scope}


\draw[factor edge] (v1) -- (v2);
\draw[factor edge] (v1) -- (v3);
\draw[factor edge] (v1) -- (v7);
\draw[factor edge] (v1) -- (v15);
\draw[factor edge] (v1) -- (v16);

\draw[factor edge] (v2) -- (v3);
\draw[factor edge] (v2) -- (v7);
\draw[factor edge] (v2) -- (v11);
\draw[factor edge] (v2) -- (v15);

\draw[factor edge] (v3) -- (v7);
\draw[factor edge] (v3) -- (v11);
\draw[factor edge] (v3) -- (v16);

\draw[factor edge] (v11) -- (v14);
\draw[factor edge] (v14) -- (v15);
\draw[factor edge] (v14) -- (v16);

\draw[factor edge] (v10) -- (v5);
\draw[factor edge] (v10) -- (v8);
\draw[factor edge] (v10) -- (v4);
\draw[factor edge] (v10) -- (v12);
\draw[factor edge] (v10) -- (v9);

\draw[factor edge] (v5) -- (v8);
\draw[factor edge] (v5) -- (v4);
\draw[factor edge] (v5) -- (v6);
\draw[factor edge] (v5) -- (v12);

\draw[factor edge] (v8) -- (v4);
\draw[factor edge] (v8) -- (v6);
\draw[factor edge] (v8) -- (v9);

\draw[factor edge] (v6) -- (v13);
\draw[factor edge] (v13) -- (v12);
\draw[factor edge] (v13) -- (v9);


\draw[product edge] (v1)  -- (v10);
\draw[product edge] (v2)  -- (v5);
\draw[product edge] (v3)  -- (v8);
\draw[product edge] (v7)  -- (v4);
\draw[product edge] (v11) -- (v6);
\draw[product edge] (v14) -- (v13);
\draw[product edge] (v15) -- (v12);
\draw[product edge] (v16) -- (v9);


\node[left vertex] at (v1)  {};
\node[left vertex] at (v2)  {};
\node[left vertex] at (v3)  {};
\node[left vertex] at (v7)  {};
\node[left vertex] at (v11) {};
\node[left vertex] at (v14) {};
\node[left vertex] at (v15) {};
\node[left vertex] at (v16) {};

\node[right vertex] at (v10) {};
\node[right vertex] at (v5)  {};
\node[right vertex] at (v8)  {};
\node[right vertex] at (v4)  {};
\node[right vertex] at (v6)  {};
\node[right vertex] at (v13) {};
\node[right vertex] at (v12) {};
\node[right vertex] at (v9)  {};

\fill[black] (v1) circle (2.8pt);
\node[left=1pt] at (v1) {$\zero$};

%
%

\end{tikzpicture}
        \caption{The graph of the irreducible slice $P$ of $\cube[5]$ in \Cref{ex:graph_fake_segment} that has a graph decomposition as $G(P)\cong \Gamma\square \cubegraph[1]$. One copy of $\Gamma$ has violet nodes (left), the other copy has green nodes (right); the edges of the $\cubegraph[1]$ factor are displayed in orange. In black, the vertex corresponding to the origin in $P$.}
        \label{fig:graph_fake_segment}
    \end{figure}
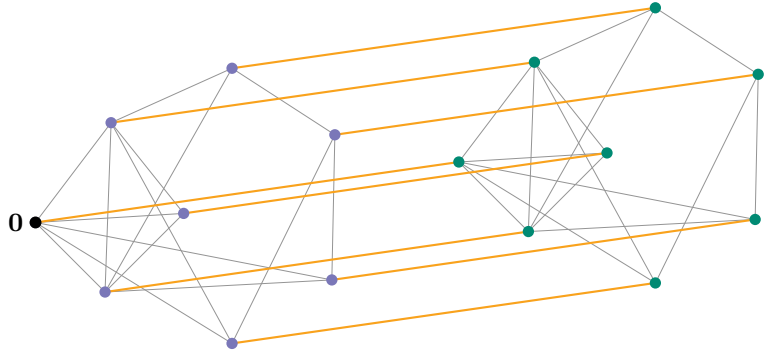
    At the origin we have $A=\{1,3\}$, $B=\{2,4,5\}$ and $C=\varnothing$.
	Thus, the origin is a non-simple vertex of degree $|A|\cdot|B|=6$.
	However, the vertex-edge graph $G(P)$ is not prime. Indeed, write $L=x_3-x_4-x_5$. Then $-2\leq L\leq 1$ for all $x \in \cube[5]$, and the equation of the slice becomes $10(x_1-x_2)+L=0$. Hence, for every fixed value of $(x_3,x_4,x_5)$, the intersection with the square in the coordinates $(x_1,x_2)$ is a non-degenerate segment. These segments give a perfect matching inside the graph $G(P)$, and this matching is a Cartesian $\cubegraph[1]$-factor.
	Therefore, $G(P)\cong \Gamma\square \cubegraph[1]$ for some graph $\Gamma$ on $8$ vertices.
	Nevertheless, this graph factor is not geometric because $P$ is irreducible.
\end{example}
	
Thanks to the following technical lemma, we can use the vertex-figure structure and the non-simple $4$-cycles to decide whether the cube factor proposed by the graph decomposition is geometric, or if it has a fake extra segment. The setting of the following statement is displayed in \cref{fig:lemma_extra_cube_factor}.
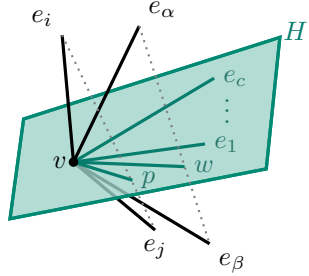
\begin{figure}[ht]
    \centering
    \begin{tikzpicture}[scale=0.6]

\definecolor{join_color}{named}{Periwinkle}
\definecolor{product_color}{named}{PineGreen}

\coordinate (O) at (0,0);

\coordinate (b1) at (1.8,-1.5);
\coordinate (b2) at (3,-1.8);

\draw[very thick] (O) -- (b1);
\draw[very thick] (O) -- (b2);

\node[below]       at (b1) {$e_{j}$};
\node[below right] at (b2) {$e_{\beta}$};

\filldraw[
    fill=white,
    fill opacity=0.7,
    draw=product_color,
    very thick
]
    (-1.4,-1.25) --
    (-1.1,0.95) --
    (4.55,2.75) --
    (4.25,-0.15) --
    cycle;
\filldraw[
    fill=product_color,
    fill opacity=0.3,
    draw=product_color,
    very thick
]
    (-1.4,-1.25) --
    (-1.1,0.95) --
    (4.55,2.75) --
    (4.25,-0.15) --
    cycle;

\coordinate (a1) at (-0.25,2.8);
\coordinate (a2) at (1.45,3);

\draw[very thick] (O) -- (a1);
\draw[very thick] (O) -- (a2);

\node[above left]  at (a1) {$e_{i}$};
\node[above right] at (a2) {$e_{\alpha}$};

\coordinate (e1) at (2.9,0.4);
\coordinate (ec) at (3.1,1.85);

\draw[product_color!80!black, very thick] (O) -- (e1);
\draw[product_color!80!black, very thick] (O) -- (ec);

\node[right] at (e1) {{\color{product_color!70!black} $e_{1}$}};
\node[right] at (ec) {{\color{product_color!70!black} $e_{c}$}};
\node[right] at (3.1,1.3) {{\color{product_color!70!black} $\vdots$}};

\coordinate (p) at (1.3,-0.4);
\coordinate (w) at (2.45,-0.1);

\draw[thick,dotted,gray] (a1) -- (b1);
\draw[thick,dotted,gray] (a2) -- (b2);
\draw[product_color!80!black, very thick] (O) -- (p);
\draw[product_color!80!black, very thick] (O) -- (w);

\node[right] at (p) {{\color{product_color!70!black} $p$}};
\node[right] at (w) {{\color{product_color!70!black} $w$}};

\fill (O) circle (3pt);

\node at (4.9,2.9) {\textcolor{product_color!80!black}{$H$}};
\node at (-0.3,0) {$v$};

\end{tikzpicture}
    \caption{Cartoon of the assumptions in \cref{lemma:c+1_4-cycle_cube_vertices}.}
    \label{fig:lemma_extra_cube_factor}
\end{figure}
\begin{lemma}\label{lemma:c+1_4-cycle_cube_vertices}
    Let $P=\cube\cap H$ be a non-simple slice with at least one simple vertex. Let $d-c$ be the rank of~$P$, so that $P \cong Q\times \cube[c]$, where $Q$ is an irreducible slice of $\cube[{d-c}]$. Suppose that the graph $G(P)$ has an accidental Cartesian cube factor of dimension $c+1$, that is, $G(P) \cong \Gamma \square \cubegraph[{c+1}]$. 

    Let $v$ be a non-simple vertex of $P$; under \ref{assumption:through_origin}, $v=\zero$ and $H=u^\perp$, and we write $A,B,C$ for the sets of indices with $u_i>0$, $u_i<0$, $u_i=0$, respectively.
    Among the $c+1$ edges at $\zero$ coming from the $\cubegraph[{c+1}]$-factor, exactly one, say $[\zero,w]$, is not an edge of the cube, and $w=e_\alpha+e_\beta$ for some $\alpha\in A$, $\beta\in B$ by \Cref{cor:shape_dominant_slice}.
    Let $p=y_{ij}$ with $i\in A\setminus\{\alpha\}$ and $j\in B\setminus\{\beta\}$. Then
    \[
    \{\zero,w,p\}\subset \text{common induced }4\text{-cycle} \qquad \Longrightarrow \qquad p=e_i+e_j\in\verts(\cube).
    \]
\end{lemma}
\begin{proof}
    Let 
    \[
    \mathfrak{C} = (\zero,w,q,p,\zero)
    \]
    be an induced $4$-cycle containing $\zero, w$, and $p$. 
    We claim that $p = e_i+e_j$. If not, then $p$ lies in the relative interior of an edge of the cube. For instance, if $u_i>-u_j$, then
    \[
    p=\frac{-u_j}{u_i}e_i+e_j,
    \]
    which lies in the relative interior of the cube edge $E=\conv(e_j,e_i+e_j)$. Every edge of $P$ incident to $p$ is obtained by intersecting $H$ with a square face of $\cube$ containing $E$. Therefore, since $p$ is adjacent to $q$, there exists an index $k\notin\{i,j\}$ such that $q\in F_1\cap H$, where 
    \[
    F_1 = \conv(e_j,e_i+e_j,e_j+e_k,e_i+e_j+e_k).
    \]
    In particular, $q_j = 1$ and $\supp(q) \subseteq \{i,j,k\}$.
    Likewise, since $q$ is adjacent to $w=e_\alpha+e_\beta$, there is a square $F_2$ of $\cube$ containing $w$ such that $[w,q]=F_2\cap H$. $F_2$ is defined by two free coordinates. Since $w_j=0$ and $q_j=1$, one of the two free coordinates of $F_2$ is $j$. Let $\ell$ be the other free coordinate. Then all coordinates of $q$, except possibly $j$ and $\ell$, agree with those of $w$.
    Now since $w,q \in F_2$ and $w_j = 0$ and $w_\alpha = w_\beta  =1$, at least one of $\alpha,\beta$ is fixed to be $1$ on $F_2$, and hence belongs to $\supp(q)$. Since $\supp(q) \subseteq \{i,j,k\}$ and $\{\alpha,\beta\}\cap\{i,j\}=\varnothing$,
    it follows that $k \in \{\alpha, \beta\}$. The other element of $\{\alpha,\beta\}$ must be the free coordinate $\ell$ of $F_2$, since otherwise it would be fixed to be equal to $1$ on $F_2$, and would therefore belong to $\supp(q)$. 
    Thus, $\{k,\ell\} = \{\alpha,\beta\}$. It follows that $i$ is not a free coordinate of $F_2$, so $w_i = 0$ implies $q_i=0$. Thus, $\supp(q) = \{k,j\}$. Since $k \in \{\alpha,\beta\}$, the only remaining possibilities are 
    \[
    q=e_\alpha+e_j \qquad\text{or}\qquad q=e_\beta+e_j.
    \]
    The second option is impossible because $u_\beta+u_j<0$ so $e_\beta+e_j\notin H$; the first option would imply $u_\alpha+u_j=0$, which is impossible by the strict dominance of $u_\alpha$, i.e., $u_\alpha > |u_j|$ (cf. \cref{cor:shape_dominant_slice}). Therefore, we have that $p=e_i+e_j \in\verts(\cube)$.    
\end{proof}

By \Cref{lemma:c_or_c+1}, the graph-theoretical cube-factor of a slice either coincides with its geometric cube-factor, or overestimates the geometric cube-factor by $1$.
The key tool to decide the actual rank of $P$ will be to compare two formulas for the number of $2$-faces of $P$ incident to a non-simple vertex. Depending on the difference of the graph-theoretic and geometric cube-factor, the two ways of counting will either differ or coincide: this will reveal the true geometric cube factor.

\begin{lemma}\label{lem:2faceformula1}
    Let $P=\cube\cap H$ be a cube slice with at least one simple vertex and at least one non-simple vertex $v$. Let $d-c$ be the rank of $P$, so that
    \[
    P \cong Q\times \cube[c],
    \]
	where $Q$ is an irreducible slice of $\cube[{d-c}]$. 
    Let $\delta=\deg(v)$ be the degree of $v$ in $G(P)$. Then the number of $2$-faces of $P$ containing $v$ is
    \[
    \mathcal{F}_2(v,c) = \frac{\delta(d-2)+c(\delta-d+1)}{2}.
    \]
\end{lemma}

\begin{proof}
    Without loss of generality, see \ref{assumption:through_origin}, we assume that $v$ is the origin~$\zero$ and $H = u^\perp$.
    Write
    \[
    A=\{i\in [d]\mid u_i>0\},\qquad
    B=\{i\in [d]\mid u_i<0\},\qquad
    C=\{i\in [d]\mid u_i=0\}.
    \]
    By \cref{prop:u-iff-irreducible}, we have $|C|=c$. Denote $a=|A|$ and $b=|B|$. Since the origin is non-simple, $a,b\geq 2$.
	
    The edges of $P$ through the origin are precisely the cube edges $[\zero,e_k]$, for $k\in C$, together with the edges $[\zero,y_{ij}]$ obtained by intersecting $H$ with the square spanned by $e_i$ and $e_j$, for $i\in A$ and $j\in B$. Hence
    \[
    \delta=\deg(\zero)=c+ab.
    \]
	By \cref{lemma:vertex_figure}, the vertex figure at the origin has the form
    \[
    \simp[c-1]\join\bigl(\simp[{a-1}]\times \simp[{b-1}]\bigr),
    \]
	where $ab=\delta-c$ and $c+a+b=d$. Therefore, the number of $2$-faces of $P$ through the origin is the number of edges of this vertex figure, namely
    \begin{align*}
    \mathcal{F}_2(\zero,c) &= \binom{c}{2} + b\binom{a}{2} + a\binom{b}{2} + abc = \frac{c(c-1)}{2}   + \frac{ab(a-1)}{2}
       + \frac{ab(b-1)}{2}
       + abc \\
    &= \frac{c(c-1) + ab(a+b-2) + 2abc}{2} = \frac{c(c-1) + ab(d+c-2)}{2} \\
    &= \frac{c(c-1) + (\delta-c)(d+c-2)}{2} = \frac{\delta(d-2) + c(\delta-d+1)}{2}.\qedhere
    \end{align*}
\end{proof}

We can exploit the formula for the number of $2$-faces of a slice of given rank at a non-simple vertex. This will allow us to detect whether the proposed graph decomposition is geometric or not.

\begin{proposition}\label{prop:rank_from_graph}
    Let $P=\cube\cap H$ be a non-simple cube slice with at least one simple vertex. Then the rank of $P$ is completely determined from its graph.
\end{proposition}
\begin{proof}
    Let $t$ be the largest integer satisfying $G(P) =\Gamma\square \cubegraph[t]$. By \cref{lemma:c_or_c+1}, the rank of $P$ is either $d-t$ or $d-t+1$. If $t=0$, then $P$ is irreducible and there is nothing to decide. If $t\geq 1$, at least $t-1$ of the $\cubegraph[1]$-factors arise from an actual geometric factor $\cube[t-1]$ of $P$. After factoring out this geometric cube, the only remaining ambiguity is therefore whether the last $\cubegraph[1]$-factor is also geometric or is accidental. Hence, without loss of generality we assume that $t=1$, and we have to decide whether $P$ is irreducible or of rank $d-1$.

    For a non-simple vertex $v\in \verts(P)$, define $\mathcal C(v,0)$ to be the sum of the numbers
    \begin{itemize}
    \item of $3$-cycles of $G(P)$ containing $v$,
    \item of induced $4$-cycles of $G(P)$ containing $v$ and at least one simple vertex,
    \item and of $P_3^*$-convex induced $5$-cycles of $G(P)$ containing $v$ and $4$ simple vertices.
    \end{itemize}
    By \cref{prop:2-faces_red}, $\mathcal C(v,0)$ would be the number of $2$-faces of $P$ incident to $v$ if $P$ were irreducible. In that case, $\mathcal C(v,0)$ would coincide with $\mathcal{F}_2(v,0)$, the number of $2$-faces of $P$ containing $v$ as computed in \Cref{lem:2faceformula1}.

    Define now $\mathcal C(v,1)$ to be $\mathcal C(v,0)$ plus the number
    \begin{itemize}
    \item of induced $4$-cycles $\mathfrak{C}=(v,v_1,v_2,v_3,v)$ of $G(P)$ containing $v$, whose vertices are all non-simple and such that $\{v,v_1\}$ and $\{v_2,v_3\}$ are contained in two different copies of $\Gamma$.
    \end{itemize}
    By \cref{prop:2-faces_red}, $\mathcal C(v,1)$ would be the number of $2$-faces of $P$ incident to $v$ if $P$ were reducible of rank $d-1$. In that case, $\mathcal C(v,1)$ would coincide with $\mathcal{F}_2(v,1)$.

    We show that $P$ is reducible of rank $d-1$ if and only if $\mathcal F_2(v,1) = \mathcal C(v,1)$.
    To this end, our next step is to compare $\mathcal F_2(v,1)-\mathcal F_2(v,0)$ with $\mathcal C(v,1)-\mathcal C(v,0)$, namely the number of induced $4$-cycles containing $v$ whose vertices are all non-simple and that have edges in two different copies of $\Gamma$.
    Assume that $P$ has a non-simple vertex $v$. Without loss of generality, see \ref{assumption:through_origin}, $v=\zero$ and $H = u^\perp$.
    As usual, we write
    \[
        A=\{i\in [d]\mid u_i>0\},\qquad
        B=\{i\in [d]\mid u_i<0\},\qquad
        C=\{i\in [d]\mid u_i=0\},
    \]
    with respective cardinalities $a,b,c$. Since the origin is non-simple, we have $a,b\geq 2$.
    Because $\deg(\zero) = c+ab$ and $d=a+b+c$, we can rewrite
    \begin{equation}\label{eq:diffF}
        \mathcal{F}_2(\zero,1)-\mathcal{F}_2(\zero,0) = \frac{\deg(\zero)-d+1}{2}= \frac{(ab+c)-(a+b+c)+1}{2} =\frac{(a-1)(b-1)}{2}.
    \end{equation}
    This is the difference in the number of $2$-faces through the origin between the reducible and irreducible cases.
    We want to compare this with the formula we get for $\mathcal C(v,1)-\mathcal C(v,0)$. 
    
    We claim that if $P$ is irreducible, then $\mathcal{F}_2(\zero,1)\neq \mathcal C(\zero,1)$. Assume irreducibility and consider the edge of $G(P)=\Gamma \square \cubegraph[1]$ incident to $\zero$ that arises from the $\cubegraph[1]$-factor. 
    The corresponding neighbor of $\zero$ in $P$ has the form $y_{\alpha\beta}$ for some $\alpha\in A$ and $\beta\in B$. Moreover,
    by \cref{cor:shape_dominant_slice}, we have $y_{\alpha \beta} =  e_\alpha+e_\beta$. Fix any pair $i\in A\setminus\{\alpha\}$ and $j\in B\setminus\{\beta\}$.
    Since $[\zero,y_{\alpha\beta}]$ is an edge from the $\cubegraph[1]$-factor and $[\zero,y_{ij}]$ is not, the Cartesian product structure forces these two edges to span a product $4$-cycle in $G(P)$.
	By \cref{lemma:c+1_4-cycle_cube_vertices}, then $y_{ij}=e_i+e_j$, which implies that $u_i=-u_j$ for all $i\in A\setminus\{\alpha\}$ and $j\in B\setminus\{\beta\}$.
    In other words, for every pair
    \[
    (i,j)\in \bigl(A\setminus\{\alpha\}\bigr) \times \bigl(B\setminus\{\beta\}\bigr),
    \]
    we obtain an induced $4$-cycle whose vertices are all non-simple
    \[
    (\zero, e_i+e_j, e_\alpha+e_\beta+e_i+e_j, e_\alpha+e_\beta, \zero),
    \]
    where the edges $[\zero,e_i+e_j]$ and $[e_\alpha+e_\beta+e_i+e_j, e_\alpha+e_\beta]$ belong to different copies of $\Gamma$ in the Cartesian product. There are exactly $(a-1)(b-1)$ such cycles. This gives
    \begin{equation}\label{eq:diffC}
        \mathcal C(\zero,1)-\mathcal C(\zero,0)=(a-1)(b-1).
    \end{equation}
    Therefore, since we assumed $P$ irreducible, we have 
    \begin{align*}
         \mathcal{F}_2(\zero,0)&= \mathcal C(\zero,0),\\
         \mathcal{F}_2(\zero,1)&\overset{\eqref{eq:diffF}}{=} \mathcal{F}_2(\zero,0)+\frac{(a-1)(b-1)}{2}= \mathcal C(\zero,0)+\frac{(a-1)(b-1)}{2}\overset{\eqref{eq:diffC}}{=}\mathcal C(\zero,1)-\frac{(a-1)(b-1)}{2}\neq \mathcal C(\zero,1)
    \end{align*}
    where the last inequality comes from the fact that $a,b\geq 2$.

    This contrasts with the reducible case: if $P$ is reducible of rank $d-1$, then $\mathcal{F}_2(\zero,1)=\mathcal C(\zero,1)$, as observed above. Hence the rank of $P$ can be read off from the graph $G(P)$: $P$ is a reducible slice of rank~$d-1$ if and only if $\mathcal{F}_2(\zero,1)=\mathcal C(\zero,1)$.
\end{proof}

\begin{example}
    We continue \cref{ex:graph_fake_segment} to showcase how to detect the correct rank of a slice. Recall that the graph decomposes as $G(P) = \Gamma \square \cubegraph[1]$. Therefore, by \cref{lemma:c_or_c+1}, the geometric cube factor has dimension $c\in\{0,1\}$. Compute the numbers of $2$-faces adjacent to the origin in both cases:
    \[
    \mathcal{F}_2(\zero,c=0)=\frac{6(5-2)}{2}=9, \qquad \mathcal{F}_2(\zero,c=1)=\frac{6(5-2)+(6-5+1)}{2}=10. 
    \]
    On the other hand, going through $G(P)$ in \cref{fig:graph_fake_segment} and counting the admissible $3$-, $4$-, $5$-cycles, one has
    \[
    \mathcal{C}(\zero,c=0)= 5+4+0 = 9 , \qquad \mathcal{C}(\zero,c=1)= 5+(4+2) + 0 = 11.
    \]
    The two extra $4$-cycles counted in $\mathcal{C}(\zero,1)$ arise from the two non-simple neighbors of the origin in $\Gamma$. These edges turn into induced $4$-cycles with all non-simple vertices because of the product structure of the graph.
    Since $\mathcal{F}_2(\zero,1)  \neq \mathcal{C}(\zero,1)$, $c=0$ and $P$ is indeed irreducible.
\end{example}

\subsection{Reconstructing cube slices from their graph}
We are now ready to prove the main theorem about recovering polytopes from their graphs.
Our result is formulated as class reconstructibility up to combinatorial equivalence, see \cref{subsec:polytopes_graphs}. In some cases, however, we obtain the stronger property of labeled class reconstructibility.
\begin{figure}[ht]
    \centering
    \begin{tikzpicture}[
    node distance=1.0cm and 1.0cm,
    inout/.style={
        draw, very thick,
        fill=white!97!Periwinkle,
        rounded corners,
        align=center,
        inner sep=5pt,
        text width=3.5cm
    },
    box/.style={
        draw,
        fill=white!97!YellowOrange,
        rounded corners,
        align=center,
        inner sep=5pt,
        text width=4.3cm
    },
    smallbox/.style={
        draw,
        fill=white!97!PineGreen,
        rounded corners,
        align=center,
        inner sep=4pt,
        text width=3.5cm,
        font=\small
    },
    arrow/.style={->, thick}
    ]
    
    \node[inout] (input) {
        \textsc{Input}: the graph\\
        $G(P)$
    };
    
    \node[box, right=of input] (dimension) {
        Recover the ambient\\
        dimension $d$\\
        \scriptsize
        \cref{prop:dimension}
    };
    
    \node[box, right=of dimension] (hyper) {
        Separate the\\
        hypersimplex case
    };
    
    \node[box, below=of dimension] (product) {
        Recover the maximal\\
        $c$ such that\\
        $P \cong Q \times \cube[c]$\\
        \scriptsize
        \cref{prop:rank_from_graph}
    };
    
    \node[box, below=of product] (local) {
        Recover edge-labeled\\
        vertex figure
    };
    
    \node[smallbox, above right=-0.3cm and 0.5cm of local] (simple) {
        At simple vertices:
        \\
        \small vertex figure is a simplex 
    };
    
    \node[smallbox, below right=-0.3cm and 0.5cm of local] (nonsimple) {
        At non-simple vertices:
        \\
        \small 
        vertex figure is reconstructible from the cycles\\
        \scriptsize
        \cref{prop:2-faces_red}
    };
    
    \node[box, below= of local] (joswig) {
        Apply \cite[Theorem 2.3]{Joswig00:ReconstructNonSimplePolytope}
    };
    
    \node[inout, left=of joswig] (output) {
        \textsc{Output}: reconstruct\\
        the combinatorial type\\
        of $P$
    };
    
    \draw[arrow] (input) -- (dimension);
    \draw[arrow] (dimension) -- (hyper);
    \draw[arrow] (dimension) -- (product);
    \draw[arrow] (product) -- (local);
    \draw[arrow] ([yshift=2mm]local.east) -- ([yshift=-4mm]simple.west);
    \draw[arrow] ([yshift=-2mm]local.east) -- ([yshift=8mm]nonsimple.west);
    \draw[arrow] (local) -- (joswig);
    \draw[arrow] (joswig) -- (output);
    \draw[arrow,dashed] (hyper) -- (output);
    
    \node[box, below=of dimension] (product) {
        Recover the maximal\\
        $c$ such that\\
        $P \simeq Q \times \cube[c]$\\
        \scriptsize
        \cref{prop:rank_from_graph}
    };
    \node[box, below=of product] (local) {
        Recover edge-labeled\\ 
        vertex figure
    };
\end{tikzpicture}
    \caption{Structure of the proof of \cref{thm:mainthm}.}
    \label{fig:proof_scheme}
\end{figure}
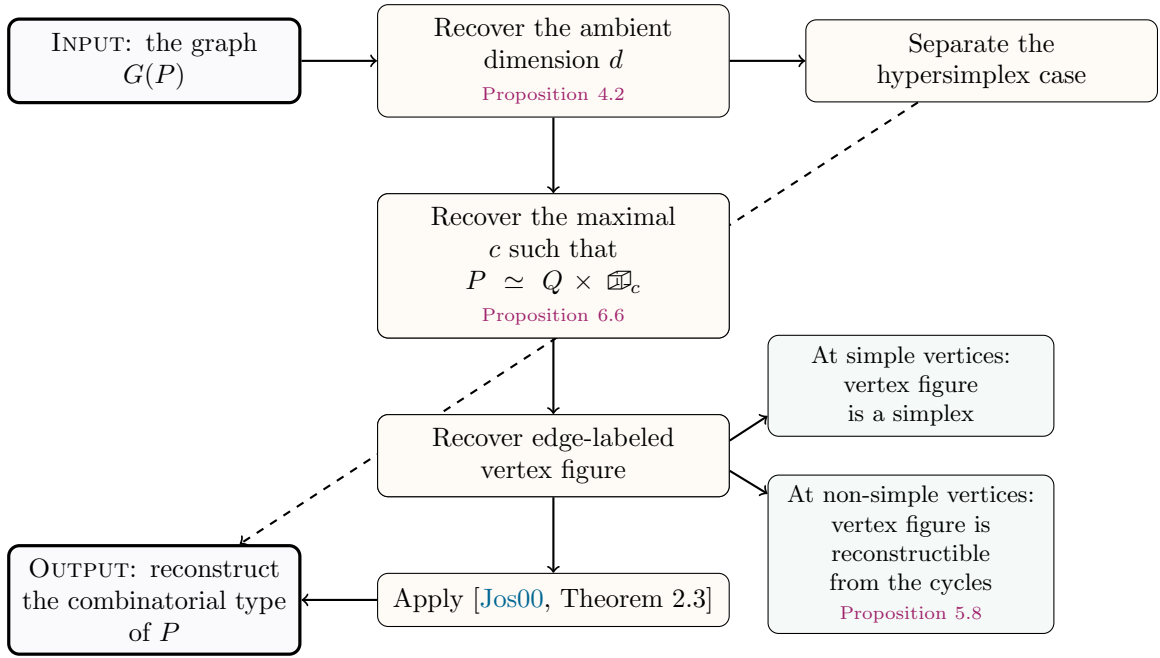	
    
\begin{proof}[Proof of \Cref{thm:mainthm}]
Let $G(P)$ be the vertex-edge graph of $P = \cube\cap H$. By \cref{prop:dimension}, $G(P)$ determines the dimension $d-1$ of $P$. We distinguish two cases.

\noindent
\underline{Case 1}: $G(P)$ is isomorphic to the Cartesian product $J(d-c,k) \square \cubegraph[c]$ of a Johnson graph with a cube graph. By \cref{prop:dimension}, then $P \cong \Delta(d-c,k)\times \cube[c]$.

\noindent
\underline{Case 2}: $G(P)$ is not isomorphic to a product of a Johnson graph with a cube graph.
According to \cref{thm:edgelabeledreconstruction}, we only need to prove that we can reconstruct the edge labeled vertex figures of a cube slice $P=\cube \cap H$ from its graph. Let $v$ be a vertex of $P$.
If $v$ is simple, then the vertex figure of $P$ at $v$ is a simplex and its edge labeled vertex figure consists of all $(d-2)$-subsets of edges incident to $P$. 
If $v$ is not simple, since $P$ has at least a simple vertex, we can determine the dimension and the rank $r$ of $P$ using \cref{prop:rank_from_graph}. Once the rank is known, \cref{prop:2-faces_red} reconstructs all the $2$-faces of $P$ incident to $v$. 

There is one possible ambiguity in this reconstruction. If
\[
G(P)\cong \Gamma\square\cubegraph[d-r+1] \qquad\text{while}\qquad P\cong Q\times\cube[d-r],
\]
then one must choose which of the $\cubegraph[1]$-factors is the accidental, non-geometric one. Different choices produce combinatorially equivalent reconstructions, but they need not preserve the original labeling of the vertices. Thus, in this case, the argument reconstructs $P$ up to combinatorial equivalence, but not necessarily as a labeled polytope.

The $2$-faces of $P$ incident to $v$ determine the labeled graph of the vertex figure $\vf_P(v)$: its vertices correspond to the edges of $P$ incident to $v$, and two of them are adjacent precisely when the corresponding edges lie in a common $2$-face through $v$. By \cref{cor:simplexslices}, such vertex figure is labeled class reconstructible, and in particular this recovers the edge-labeled vertex figure of $P$ at $v$. By \cref{thm:edgelabeledreconstruction} the claim follows.  
\end{proof}

\begin{remark}
The preceding proof reconstructs the combinatorial type of $P$ from its graph in all cases. Moreover, except for one explicitly detectable situation, the reconstruction is in fact labeled.
Indeed, the only choice made in the proof arises when $r<d$ and
\[
G(P)\cong \Gamma\square\cubegraph[d-r+1], \qquad P\cong Q\times\cube[d-r],
\]
so that the graph contains one more $\cubegraph[1]$-factor than the geometric cube factor of $P$. In this case, one must choose which $\cubegraph[1]$-factor is the accidental one, and different choices may lead to different identifications of the reconstructed faces with the original labeled vertex set.
This exceptional situation can be recognized from the graph, since both the rank $r$ of $P$ and the maximal Cartesian cube factor of $G(P)$ are graph-theoretically determined. If the dimension of the maximal cube factor is exactly $d-r$, then every cube factor is geometric and no choice occurs in the reconstruction. Consequently, outside the case where the maximal cube factor has dimension $d-r+1$, the graph determines the face lattice of $P$ on its given vertex set; that is, the reconstruction is labeled.
\end{remark}

The next example shows that, when $P$ has rank $r$ but
\[
G(P)\cong \Gamma\square\cubegraph[d-r+1],
\]
the graph may determine the combinatorial type of $P$ without determining its labeled face lattice.
\begin{figure}[ht]
    \centering
    \begin{tikzpicture}[scale=0.8, line join=round, line cap=round]

\tikzset{
  factor edge/.style={draw=black!42, line width=0.35pt},
  product edge A/.style={draw=YellowOrange, line width=0.8pt},
  product edge B/.style={draw=BrickRed, line width=0.8pt},
  vertex/.style={circle, fill=Periwinkle, inner sep=1.5pt},
}


\coordinate (v3)  at (1.35,0.00);   
\coordinate (v1)  at (1.00,2.25);   
\coordinate (v10) at (2.35,1.75);   
\coordinate (v12) at (0.00,1.05);   

\begin{scope}[xshift=5.2cm,yshift=0.5cm]
\coordinate (v16) at (1.35,0.00);   
\coordinate (v6)  at (1.00,2.25);   
\coordinate (v8)  at (2.35,1.75);   
\coordinate (v9)  at (0.00,1.05);   
\end{scope}

\begin{scope}[xshift=0.2cm,yshift=4cm]
\coordinate (v2)  at (1.35,0.00);   
\coordinate (v4)  at (1.00,2.25);   
\coordinate (v11) at (2.35,1.75);   
\coordinate (v13) at (0.00,1.05);   
\end{scope}

\begin{scope}[xshift=5.4cm,yshift=4.5cm]
\coordinate (v15) at (1.35,0.00);   
\coordinate (v5)  at (1.00,2.25);   
\coordinate (v7)  at (2.35,1.75);   
\coordinate (v14) at (0.00,1.05);   
\end{scope}


\draw[factor edge] (v3) -- (v1);
\draw[factor edge] (v3) -- (v10);
\draw[factor edge] (v3) -- (v12);
\draw[factor edge] (v1) -- (v10);
\draw[factor edge] (v1) -- (v12);

\draw[factor edge] (v16) -- (v6);
\draw[factor edge] (v16) -- (v8);
\draw[factor edge] (v16) -- (v9);
\draw[factor edge] (v6) -- (v8);
\draw[factor edge] (v6) -- (v9);

\draw[factor edge] (v2) -- (v4);
\draw[factor edge] (v2) -- (v11);
\draw[factor edge] (v2) -- (v13);
\draw[factor edge] (v4) -- (v11);
\draw[factor edge] (v4) -- (v13);

\draw[factor edge] (v15) -- (v5);
\draw[factor edge] (v15) -- (v7);
\draw[factor edge] (v15) -- (v14);
\draw[factor edge] (v5) -- (v7);
\draw[factor edge] (v5) -- (v14);


\draw[product edge A] (v3)  -- (v16);
\draw[product edge A] (v1)  -- (v6);
\draw[product edge A] (v10) -- (v8);
\draw[product edge A] (v12) -- (v9);

\draw[product edge A] (v2)  -- (v15);
\draw[product edge A] (v4)  -- (v5);
\draw[product edge A] (v11) -- (v7);
\draw[product edge A] (v13) -- (v14);


\draw[product edge B] (v3)  -- (v2);
\draw[product edge B] (v1)  -- (v4);
\draw[product edge B] (v10) -- (v11);
\draw[product edge B] (v12) -- (v13);

\draw[product edge B] (v16) -- (v15);
\draw[product edge B] (v6)  -- (v5);
\draw[product edge B] (v8)  -- (v7);
\draw[product edge B] (v9)  -- (v14);


\node[vertex]  at (v3)  {};
\node[vertex]  at (v1)  {};
\node[vertex]  at (v10) {};
\node[vertex]  at (v12) {};

\node[vertex]  at (v2)  {};
\node[vertex]  at (v4)  {};
\node[vertex]  at (v11) {};
\node[vertex]  at (v13) {};

\node[vertex] at (v16) {};
\node[vertex] at (v6)  {};
\node[vertex] at (v8)  {};
\node[vertex] at (v9)  {};

\node[vertex] at (v15) {};
\node[vertex] at (v5)  {};
\node[vertex] at (v7)  {};
\node[vertex] at (v14) {};


\node[below right=1pt] at (v3)  {\scriptsize $1$};
\node[above left=1pt]  at (v1)  {\scriptsize $2$};
\node[right=1pt]       at (v10) {\scriptsize $3$};
\node[below left=1pt]  at (v12) {\scriptsize $4$};

\node[below right=1pt] at (v2)  {\scriptsize $5$};
\node[above left=1pt]  at (v4)  {\scriptsize $6$};
\node[right=1pt]       at (v11) {\scriptsize $7$};
\node[below left=1pt]  at (v13) {\scriptsize $8$};

\node[below right=1pt] at (v16) {\scriptsize $9$};
\node[above left=1pt]  at (v6)  {\scriptsize $10$};
\node[right=1pt]       at (v8)  {\scriptsize $11$};
\node[below left=1pt]  at (v9)  {\scriptsize $12$};

\node[below right=1pt] at (v15) {\scriptsize $13$};
\node[above left=1pt]  at (v5)  {\scriptsize $14$};
\node[right=1pt]       at (v7)  {\scriptsize $15$};
\node[below left=1pt]  at (v14) {\scriptsize $16$};

\end{tikzpicture}
    \caption{The vertex-edge graph of $P$ from \cref{ex:ambiguous_recovery}. It has two Cartesian $\cubegraph[1]$-factors, shown in orange and red, although only one of them arises from a geometric segment factor of $P$.}
    \label{fig:ambiguous_recovery}
\end{figure}
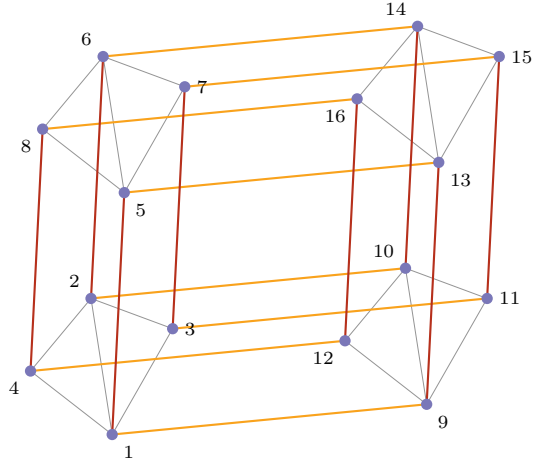
\begin{example}\label{ex:ambiguous_recovery}
    Let $P = \cube[5] \cap H$ with $H = \set{x\in \R^5}{ 10 x_1 - 10 x_2 + x_3 - x_4 = 0}$. Since the equation of $H$ does not involve $x_5$, $P$ has rank $4$. On the other hand, the Cartesian prime factorization of its graph is
    \[
    G(P) \cong (K_4 - e) \square \cubegraph[2],
    \]
    where $(K_4 - e)$ denotes the complete graph on four vertices with one edge removed, see \cref{fig:ambiguous_recovery}. 
    The reconstruction therefore requires a choice of which of the two $\cubegraph[1]$-factors is to be interpreted as the geometric factor. Either choice produces a polytope of the same combinatorial type, namely
    \[
    P\cong Q \times \cube[1],
    \]
    where $Q$ is obtained by gluing two triangular prisms along a quadrilateral facet. However, the two choices induce different face structures on the given labeled vertex set. More precisely, if the yellow $\cubegraph[1]$-factor is chosen as geometric, then
    \[
    \{1,2,5,6\}
    \]
    is not the vertex set of a face, but it is the vertex set of a face if the red factor is chosen as geometric. Thus the graph determines the combinatorial type of $P$, but not which subsets of its labeled vertices are faces.
\end{example}

\section{Generalizations and future work}\label{sec: generalizations}

The reconstruction result of \cref{thm:mainthm} naturally raises the question of how far the underlying ideas extend beyond hyperplane sections of cubes. The proof relies both on the special combinatorial structure of the cube and on the fact that the section has codimension one, and it is therefore not immediate which parts of the argument survive in a more general setting.
We conclude by briefly discussing two directions that seem particularly natural. As we will see, in both directions, our results are negative and highlight the exceptionality of cube slices.

First, one may replace hyperplane sections by intersections of the cube with affine subspaces of higher codimension. In this case, the local structure of the resulting polytope becomes more complicated, and the relation between its faces and the faces of the ambient cube is less rigid. It would be interesting to determine whether the vertex-edge graph still determines the combinatorial type. As we will see in~\cref{thm:lowerdimsections}, there is no reconstructibility for sections with subspaces of codimension at least~$4$, leaving open the cases of codimension~$2$ and~$3$.

Second, one may ask whether analogous reconstruction results hold when the cube is replaced by a broader class of polytopes. As we showed in the introduction, it easy to see that one cannot reconstruct slices of arbitrary polytopes from their graph. Here the main challenge is to identify which combinatorial features of the cube are essential for the argument, and which can be formulated in a way that applies to other ambient polytopes, such as products of simplices, zonotopes, or more general simple polytopes. The latter is a particularly natural family, but we will see in \cref{thm:slice_simple_polytope} that there is no reconstruction for simple polytopes either.

\subsection{Higher codimension} \label{sec:highcodim}

In our first example, we show that we can embed two $4$-polytopes with the same graph and different combinatorial type as sections of the $8$-dimensional cube with $4$-dimensional flats. 
\begin{example}\label{ex:4-dim_sections}
Consider the $4$-polytope $P\subset \mathbb{R}^4$ with vertices
\begin{gather*}
v_1=(0,0,1,0), \quad
v_2=(0,0,-1,0), \quad
v_3=(1,0,0,0),\\
v_4=(0,1,0,0), \quad
v_5=(-1,-1,0,0), \quad
v_6=(0,0,0,1).
\end{gather*}
This is a pyramid over a triangular bipyramid.
The facets of $P$ are cut out by the following nonnegative affine functions:
\begin{gather*}
s_1=x_4,
\
s_2=\tfrac{1}{3}(1-x_1-x_2-x_3-x_4),
\
s_3=\tfrac{1}{3}(1+2x_1-x_2-x_3-x_4),
\
s_4=\tfrac{1}{3}(1-x_1+2x_2-x_3-x_4),\\
s_5=\tfrac{1}{3}(1-x_1-x_2+x_3-x_4),
\qquad
s_6=\tfrac{1}{3}(1+2x_1-x_2+x_3-x_4),
\qquad
s_7=\tfrac{1}{3}(1-x_1+2x_2+x_3-x_4).
\end{gather*}
Each $s_i$ is a normalized slack: it is nonnegative on $P$, it vanishes exactly on one facet, and it takes values in $[0,1]$ on all of $P$.
Thus the map
\[
\varphi_P: x\mapsto (s_1(x),s_2(x),\ldots,s_7(x))
\]
embeds $P$ as an affine section of $\cube[7]$, i.e., $\varphi_P(P) = [0,1]^7 \cap \varphi_P(\R^4)$. In order to put it inside $\cube[8]$, we add one more affine coordinate, such as $s_8=\frac12-\frac{x_3}{4}+\frac{x_4}{2}$. This coordinate also takes values in $[0,1]$ on $P$. Hence we obtain an affine embedding
\[
\Phi_P:P\longrightarrow \cube[8]
\]
given by $\Phi_P(x)=(s_1(x),s_2(x),\ldots,s_8(x))$. The image $\Phi_P(P)$ is equal to $\cube[8]\cap L_P$, where $L_P$ is the affine hull of the six points $\Phi_P(v_i)$. The affine subspace $L_P$ is cut out by the equations
\begin{gather*}
2y_1-y_2+2y_3+2y_4+3y_5=2, \qquad 2y_1+2y_2-y_3+2y_4+3y_6=2,\\
2y_1+2y_2+2y_3-y_4+3y_7=2, \qquad 3y_1+y_2+y_3+y_4-4y_8=-1.
\end{gather*}
Now define a second $4$-polytope $Q\subset \mathbb{R}^4$ with vertices
\begin{gather*}
w_1=\left(0,0,1,\frac12\right), \quad
w_2=(0,0,-1,0), \quad
w_3=(1,0,0,0),\\
w_4=(0,1,0,0), \quad
w_5=(-1,-1,0,0), \quad
w_6=(0,0,0,1).
\end{gather*}
This polytope is obtained from $P$ by moving one vertex. Combinatorially, the non-simplex facet of $P$ is replaced by two simplex facets, but the graph of $P$ coincides with the graph of $Q$.
The facets of $Q$ are cut out by the following normalized slack functions:
\begin{gather*}
t_1=x_4-\tfrac{x_3}{2},
\qquad  
t_2=x_4,
\qquad \quad  \ 
t_3=\tfrac{1}{3}(1-x_1-x_2-\tfrac{x_3}{2}-x_4),
\quad
t_4=\tfrac{1}{3}(1+2x_1-x_2-\tfrac{x_3}{2}-x_4),\\
t_5=\tfrac{1}{3}(1-x_1+2x_2-\tfrac{x_3}{2}-x_4),
\quad
t_6=\tfrac{1}{3}(1-x_1-x_2+x_3-x_4),
\quad
t_7=\tfrac{1}{3}(1+2x_1-x_2+x_3-x_4), \\
\qquad
t_8=\tfrac{1}{3}(1-x_1+2x_2+x_3-x_4). \hspace{1.5em}
\end{gather*}
Again, each $t_i$ is nonnegative on $Q$, vanishes exactly on one facet, and takes values in $[0,1]$.
The slack map
\[
\Phi_Q:Q\longrightarrow \cube[8]
\]
is given by $\Phi_Q(x)=(t_1(x),t_2(x),\ldots,t_8(x))$. 
Its image is an affine section of $\cube[8]$. More precisely, it is $\cube[8]\cap L_Q$, where $L_Q$ is the affine hull of the six points $\Phi_Q(w_i)$.
The affine subspace $L_Q$ is in particular cut out by the equations
\begin{gather*}
y_1-2y_2-y_3-y_4-y_5=-1,
\qquad
y_1-y_2-y_3+y_6=0,\\
y_1-y_2-y_4+y_7=0,
\qquad
y_2+y_3+y_4+y_8=1.
\end{gather*}
In this way, both $P$ and $Q$ are realized as $4$-dimensional affine sections of the same cube $\cube[8]$. The construction is simply the slack-coordinate embedding, so it holds analogously in more general examples.
\end{example}

The preceding construction can be stabilized by taking Cartesian products with cubes and by fixing additional coordinates. This produces counterexamples in every ambient dimension $d\geq 8$ and for every section dimension between $4$ and $d-4$.
\begin{theorem} \label{thm:lowerdimsections}
    For every $d\geq 8$ and for every $4\leq k\leq d-4$ there exist two sections of $\cube$ of dimension $k$ with isomorphic vertex-edge graph but non-isomorphic face lattice.
\end{theorem}
\begin{proof}
    Let $P$ and $Q$ be the two $4$-dimensional sections of $\cube[8]$ from \Cref{ex:4-dim_sections}, which have isomorphic vertex-edge graphs but non-isomorphic face lattices. For $4\leq k\leq d-4$, set
    \[
    P_k=P\times \cube[k-4]\times\{1/2\}^{d-k-4}, \qquad Q_k=Q\times \cube[k-4]\times\{1/2\}^{d-k-4}.
    \]
    These are $k$-dimensional sections of
    \[
    \cube[8]\times \cube[k-4]\times \cube[d-k-4]\cong \cube.
    \]
    Since the graph of a Cartesian product of polytopes is the Cartesian product of their graphs, $G(P_k)\cong G(Q_k)$. On the other hand, taking the product with the same cube factor does not make the non-isomorphic face lattices of $P$ and $Q$ isomorphic. Hence $P_k$ and $Q_k$ are not combinatorially equivalent.
\end{proof}

Thus, graph reconstruction already fails for cube sections of codimension at least four. The remaining cases not covered by this construction are therefore codimensions two and three.

\begin{question}
Can slices of the cube of codimension two and three be recovered from their graphs?
\end{question}

\subsection{Slices of other polytopes}
Our reconstruction result cannot be extended to slices of arbitrary ambient polytopes. Indeed, it is easy to construct a (non-simple) polytope having two facets with isomorphic vertex-edge graphs but non-isomorphic face lattices. Since facets are hyperplane sections, this already provides counterexamples in the absence of assumptions on the ambient polytope.

On the positive side, hyperplane sections of simplices are also reconstructible from their graphs (\cref{cor:simplexslices}). From a combinatorial viewpoint, simplices and cubes belong to broader and natural classes of simple polytopes. For example, simplices are the basic factors in products of simplices, while cubes are products of segments. This suggests investigating to what extent the reconstruction phenomenon persists in these larger classes.

\begin{question}
Are hyperplane sections of a product of simplices class reconstructible?
\end{question}

More generally, in view of the intensive use of the simplicity of cubes in our proof, it would be natural to ask if hyperplane sections of any simple polytope may be reconstructed from their vertex-edge graphs. 
However, we will see that the answer is negative.

Finding such a counterexample is substantially more delicate. Indeed, to our knowledge, the standard examples of combinatorially non-equivalent polytopes with isomorphic graphs cannot occur as hyperplane sections of simple polytopes, because they fail basic local restrictions imposed by such a realization. More precisely, let $P$ be a $(d-1)$-dimensional hyperplane section of a simple $d$-polytope. Then:
\begin{itemize}
    \item every vertex of $P$ is incident to at most $d$ facets, equivalently, to at most $\dim(P)+1$ facets; and hence 
    \item every vertex figure of $P$ is an iterated pyramid over a product of two simplices;
\end{itemize}
Thus, a counterexample to graph reconstruction in this setting would have to consist of two polytopes with isomorphic graphs and non-isomorphic face lattices, while satisfying these rather restrictive local conditions at every vertex. 

The following example constructs such a pair, which will be the starting point of our counterexample. This example is inspired by \cite[Sections 1.3 and 2.3]{PfeiflePilaudSantos}.
It combines a graph automorphism that is not induced by a combinatorial automorphism of the polytope with vertex truncation, which corresponds graph-theoretically to the star-clique operation.

\begin{example}\label{ex:simple_slices_counterexample}
Consider the four-dimensional polytope $R$ with the following set of vertices $\mathcal R=\{r_1,\dots ,r_{16}\}$: 
\begin{equation*}
\begin{array}{llll}
  r_1    \coloneqq (-1,1,1,0),                               &
  r_2    \coloneqq (0,0,0,0),                                &
  r_3    \coloneqq (0,1,0,0),                                &
  r_4    \coloneqq \left(0,\tfrac16,0,\tfrac56\right),       \\[6pt]
  r_5    \coloneqq \left(0,0,0,\tfrac12\right),              &
  r_6    \coloneqq \left(0,0,\tfrac12,0\right),              &
  r_7    \coloneqq \left(-\tfrac12,\tfrac12,1,0\right),      &
  r_8    \coloneqq \left(-\tfrac12,\tfrac12,\tfrac12,\tfrac12\right), \\[6pt]
  r_9    \coloneqq (1,1,1,0),                                &
  r_{10} \coloneqq (1,0,0,0),                                &
  r_{11} \coloneqq (1,1,0,0),                                &
  r_{12} \coloneqq \left(\tfrac18,\tfrac18,0,\tfrac78\right), \\[6pt]
  r_{13} \coloneqq \left(\tfrac14,0,0,\tfrac34\right),       &
  r_{14} \coloneqq \left(\tfrac14,0,\tfrac34,0\right),       &
  r_{15} \coloneqq \left(\tfrac16,\tfrac16,1,0\right),       &
  r_{16} \coloneqq \left(\tfrac16,\tfrac16,\tfrac16,\tfrac56\right).
\end{array}
\end{equation*}
Its vertex-edge graph $G(R)$ decomposes as $\Gamma\square \cubegraph[1]$, and is displayed in \cref{fig:simple_counterex_graph}. Apart from the $\cubegraph[1]$-factor, the graph has an additional symmetry given by the involution $\phi$ that fixes the vertices $r_1,r_5,r_9,r_{13}$ and swaps
\[
r_2 \leftrightarrow r_8, \quad r_3\leftrightarrow r_7,\quad r_4\leftrightarrow r_6, \quad 
r_{10} \leftrightarrow r_{16}, \quad r_{11}\leftrightarrow r_{15}, \quad r_{12}\leftrightarrow r_{14}.
\]
This is however not a symmetry of $R$. Indeed, consider the vertices $r_3$ and $r_7$. They are both incident to four facets, but these have different numbers of vertices: the facets incident to $r_3$ have $6,8,8,10$ vertices, the facets incident to $r_7$ have $6,6,10,10$ vertices.
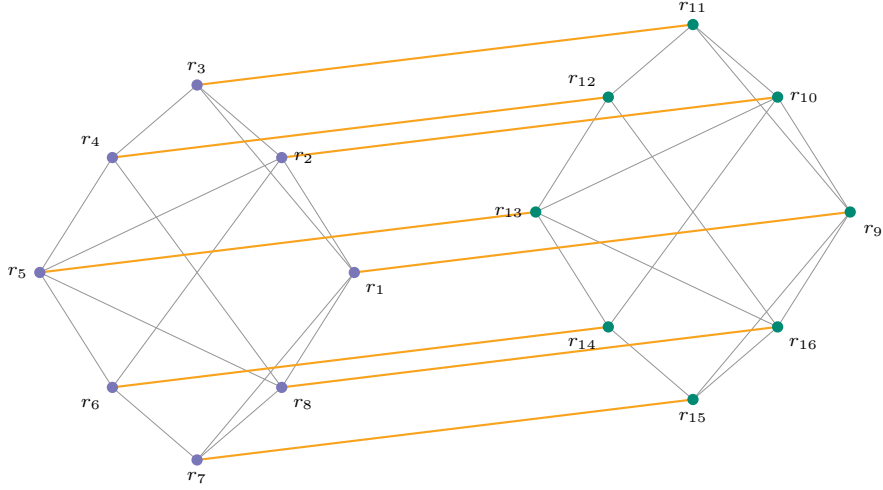
\begin{figure}[ht]
    \centering
    \begin{tikzpicture}[scale=0.8, line join=round, line cap=round]

\definecolor{left_color}{named}{Periwinkle}
\definecolor{right_color}{named}{PineGreen}

\tikzset{
  factor edge/.style={draw=black!42, line width=0.35pt},
  product edge/.style={draw=YellowOrange, line width=0.8pt},
  left vertex/.style={circle, fill=left_color, inner sep=1.5pt},
  right vertex/.style={circle, fill=right_color, inner sep=1.5pt}
}


\coordinate (v7)  at (0.0,  0.0);
\coordinate (v8)  at (1.2,  1.9);
\coordinate (v13) at (1.2, -1.9);
\coordinate (v2)  at (2.6,  3.1);
\coordinate (v12) at (2.6, -3.1);
\coordinate (v1)  at (4.0,  1.9);
\coordinate (v6)  at (4.0, -1.9);
\coordinate (v0)  at (5.2,  0.0);


\begin{scope}[xshift=8.2cm,yshift=1cm]
\coordinate (v9)  at (0.0,  0.0);
\coordinate (v10) at (1.2,  1.9);
\coordinate (v14) at (1.2, -1.9);
\coordinate (v4)  at (2.6,  3.1);
\coordinate (v15) at (2.6, -3.1);
\coordinate (v3)  at (4.0,  1.9);
\coordinate (v11) at (4.0, -1.9);
\coordinate (v5)  at (5.2,  0.0);
\end{scope}




\draw[factor edge] (v0)  -- (v1);
\draw[factor edge] (v0)  -- (v2);
\draw[factor edge] (v0)  -- (v6);
\draw[factor edge] (v0)  -- (v12);

\draw[factor edge] (v1)  -- (v2);
\draw[factor edge] (v1)  -- (v7);
\draw[factor edge] (v1)  -- (v13);

\draw[factor edge] (v2)  -- (v8);

\draw[factor edge] (v6)  -- (v7);
\draw[factor edge] (v6)  -- (v8);
\draw[factor edge] (v6)  -- (v12);

\draw[factor edge] (v7)  -- (v8);
\draw[factor edge] (v7)  -- (v13);

\draw[factor edge] (v12) -- (v13);

\draw[factor edge] (v5)  -- (v3);
\draw[factor edge] (v5)  -- (v4);
\draw[factor edge] (v5)  -- (v11);
\draw[factor edge] (v5)  -- (v15);

\draw[factor edge] (v3)  -- (v4);
\draw[factor edge] (v3)  -- (v9);
\draw[factor edge] (v3)  -- (v14);

\draw[factor edge] (v4)  -- (v10);

\draw[factor edge] (v11) -- (v9);
\draw[factor edge] (v11) -- (v10);
\draw[factor edge] (v11) -- (v15);

\draw[factor edge] (v9)  -- (v10);
\draw[factor edge] (v9)  -- (v14);

\draw[factor edge] (v15) -- (v14);


\draw[product edge] (v0)  -- (v5);
\draw[product edge] (v1)  -- (v3);
\draw[product edge] (v2)  -- (v4);
\draw[product edge] (v6)  -- (v11);
\draw[product edge] (v7)  -- (v9);
\draw[product edge] (v8)  -- (v10);
\draw[product edge] (v12) -- (v15);
\draw[product edge] (v13) -- (v14);


\node[left vertex]  at (v0)  {};
\node[left vertex]  at (v1)  {};
\node[left vertex]  at (v2)  {};
\node[left vertex]  at (v6)  {};
\node[left vertex]  at (v7)  {};
\node[left vertex]  at (v8)  {};
\node[left vertex]  at (v12) {};
\node[left vertex]  at (v13) {};

\node[right vertex] at (v5)  {};
\node[right vertex] at (v3)  {};
\node[right vertex] at (v4)  {};
\node[right vertex] at (v11) {};
\node[right vertex] at (v9)  {};
\node[right vertex] at (v10) {};
\node[right vertex] at (v15) {};
\node[right vertex] at (v14) {};




\node[below right=1pt]  at (v0)  {\scriptsize $r_1$};
\node[right=1pt]        at (v1)  {\scriptsize $r_2$};
\node[above=1pt]        at (v2)  {\scriptsize $r_3$};
\node[above left=1pt]   at (v8)  {\scriptsize $r_4$};
\node[left=1pt]         at (v7)  {\scriptsize $r_5$};
\node[below left=1pt]   at (v13) {\scriptsize $r_6$};
\node[below=1pt]        at (v12) {\scriptsize $r_7$};
\node[below right=1pt]  at (v6)  {\scriptsize $r_8$};

\node[below right=2pt]  at (v5)  {\scriptsize $r_9$};
\node[right=1pt]        at (v3)  {\scriptsize $r_{10}$};
\node[above=1pt]        at (v4)  {\scriptsize $r_{11}$};
\node[above left=1pt]   at (v10) {\scriptsize $r_{12}$};
\node[left=1pt]         at (v9)  {\scriptsize $r_{13}$};
\node[below left=1pt]   at (v14) {\scriptsize $r_{14}$};
\node[below=1pt]        at (v15) {\scriptsize $r_{15}$};
\node[below right=1pt]  at (v11) {\scriptsize $r_{16}$};

\end{tikzpicture}
    \caption{Four-dimensional polytope with two truncations (at $3$ and $7$) that give non-combinatorially equivalent polytopes with isomorphic vertex-edge graph.}
    \label{fig:simple_counterex_graph}
\end{figure}
The idea is to truncate $R$ near $r_3$ to get a polytope $P$, and near $r_7$ to get a polytope $Q$. Graph-theoretically, the truncations behave the same way hence giving
\[
G(P) \cong G(Q).
\]
Because the facets at $r_3$ and $r_7$ are combinatorially different, the truncations produce two polytopes that are not combinatorially equivalent. This is achieved for instance by the polytopes with vertices
\begin{align*}
\verts({P}) &= (\mathcal R \setminus\{r_3\}) \cup \left\lbrace \left(-\tfrac{1}{10},1,\tfrac{1}{10},0\right),
\left(0,\tfrac{9}{10},0,0\right),
\left(\tfrac{1}{10},1,0,0\right),
\left(0,\tfrac{11}{12},0,\tfrac{1}{12}\right)\right\rbrace, \\
\verts(Q) &= (\mathcal R\setminus\{r_7\}) \cup \left\lbrace \left(-\tfrac{11}{20},\tfrac{11}{20},1,0\right), \left(-\tfrac12,\tfrac12,\tfrac{19}{20},\tfrac1{20}\right), \left(-\tfrac9{20},\tfrac9{20},\tfrac{19}{20},0\right), \left(-\tfrac{13}{30},\tfrac7{15},1,0\right)\right\rbrace.
\end{align*}
They both have $f$-vector $(19, 42, 34, 11)$ and isomorphic vertex-edge graphs. However, they are not combinatorially equivalent. To see this, it is enough to compute the facet list to find that $P$ has facets with $10$ vertices, whereas $Q$ does not.
\end{example}

\begin{proposition}\label{prop:embed_and_glue}
    Let $P$ and $Q$ be two $(d-1)$-dimensional polytopes such that every vertex of $P$ and $Q$ is contained in at most $d$ facets. Then there exists a simple $d$-dimensional polytope $R$ having hyperplane sections combinatorially equivalent to $P$ and $Q$.
\end{proposition}
\begin{proof}

We first show that $P$ is a hyperplane section of a simple $d$-polytope. Write
\[
  P=\bigl\{x\in\R^{d-1}:\ \langle a_i,x\rangle\le b_i,\ i=1,\dots,m\bigr\},
\]
where $a_1,\dots,a_m$ are the facet normals of $P$.
Given heights $h=(h_1,\dots,h_m)\in\mathbb R^m$, set
\[
  R=\bigl\{(x,x_d)\in\mathbb R^{d-1}\times\R:\
      \langle a_i,x\rangle+h_ix_d\le b_i,\ i=1,\dots,m\bigr\} ,
\]
so that $R\cap\{x_d=0\}=P$.
Suppose that $d+1$ of the constraints, say those indexed by $S$, are tight at a point $(x,t)$
with $t\ne0$. Solving each of the equations $\langle a_i,x\rangle+h_it=b_i$ for $h_i$ gives
\[
  h_i=\frac{b_i-\langle a_i,x\rangle}{t}=\alpha b_i+\langle\beta,a_i\rangle
  \qquad (i\in S),\qquad \alpha=\frac1t\ne0,\quad \beta=-\frac xt .
\]
For a fixed $S$, the heights of this form constitute a linear subspace $L_S$ of $\R^m$. $L_S$ is a proper subspace of $\R^m$
because $(\alpha,\beta)$ runs over only $d < |S| = d+1$ parameters, while the remaining coordinates are free. As there are finitely
many such $S$, we may fix an $h \in \R^m$ lying in none of the subspaces $L_S, S \in \binom{[m]}{d+1}$.
Then, by construction of $h$, no point $(x,t) \in R$ with $t\neq 0$ belongs to more than $d$ facets of $R$. This is also true for $t=0$, because any point $(x,0) \in R$ satisfies $x \in P$, which is contained in at most $d$ facets by assumption.

We have shown that $R$ is simple, however, it may be unbounded. To address this, let $\Delta$ be a $d$-simplex containing $P$ in its interior, generic enough that no
$d+1$ of the hyperplanes bounding $R$ and $\Delta$ have a common point in $R\cap\Delta$. Then
$P'=R\cap\Delta$ is a simple $d$-polytope, by construction, and since $P\subseteq\Delta$,
\[
  P'\cap\{x_d=0\}=R\cap\{x_d=0\}\cap\Delta=P .
\]
Applying the same construction to $Q$ gives a simple $d$-polytope $Q'$ having $Q$ as a
hyperplane section.
It remains to realize the two sections inside a common simple polytope. Choose vertices $v\in P'$ and $w\in Q'$ away from $P$ and $Q$. Since $P'$ and $Q'$ are simple, exactly $d$ facets meet at each of $v$ and $w$.

We first describe explicitly a projective transformation sending $v$ to infinity. Let
\[
    \langle a_i, x\rangle \leq b_i,\qquad i=1,\ldots,d,
\]
be the inequalities defining the $d$ facets of $P'$ containing $v$, and let $s_i(x)=b_i-\langle a_i, x\rangle$ be their slack functions. Then $s_i\geq 0$ on $P'$, and $s_1(x)=\cdots=s_d(x)=0$ holds on $P'$ precisely at $v$. Set
\[
    D(x)=s_1(x)+\cdots+s_d(x).
\]
Then $D(x)>0$ on $P'\setminus\{v\}$. Consider the projective transformation mapping $x\in\R^d$ to $(y,t)\in \R^{d-1}\times\R$ via
\[
    y_i=\frac{s_i(x)}{D(x)},\quad i=1,\ldots,d-1,
    \qquad
    t=\frac{1}{D(x)}.
\]
As $x$ approaches $v$, one has $D(x)\to 0$ and hence $t\to+\infty$, so this transformation sends $v$ to infinity.
The image of $P'$ under this transformation is
\[
\widetilde P' = \{(y,t):y\in\simp[d-1],\ t\geq f(y)\}
\]
where $f$ is the maximum of finitely many affine functions, each corresponding to a facet not containing $v$, and is therefore convex and piecewise linear. We apply the same construction to $Q'$ at $w$. After reflecting the last coordinate, its image has the form
\[
    \widetilde Q' = \{(y,t):y\in\simp[d-1],\ t\leq g(y)\},
\]
where $g$ is concave and piecewise linear. 

For $L>0$, translate $\widetilde Q'$ by $L \cdot e_d$, and denote the translated polyhedron by $\widetilde Q'_L$. Let $R_L = \widetilde P'\cap\widetilde Q'_L$. Since $\simp[d-1]$ is compact, for $L$ sufficiently large we have $f(y)<g(y)+L$ for every $y\in\simp[d-1]$. Thus the finite lower boundary coming from $P'$ and the finite upper boundary coming from $Q'$ are disjoint. In particular, the intersection does not cut any part of either finite boundary. The vertices of $R_L$ are precisely the finite vertices of $\widetilde P'$ and $\widetilde Q'_L$, namely the vertices of $P'$ different from $v$ and the vertices of $Q'$ different from $w$. Thus, $R_L$ is simple.

Since projective transformations send hyperplanes to hyperplanes, and the slice $P$ does not contain $v$, then the projective transformation sends $P$ to 
\[
\widetilde P=\widetilde P'\cap H_P,
\]
for some hyperplane $H_P$. Because we chose $L$ sufficiently large, we can ensure that the compact set $\widetilde P$ is entirely contained in $\widetilde Q'_L$. Therefore,
\[
    R_L\cap H_P = \widetilde P'\cap\widetilde Q'_L\cap H_P = \widetilde P'\cap H_P = \widetilde P\cong P.
\]
The same argument applies to $Q$. Taking $R=R_L$ for $L$ sufficiently large proves the claim.
\end{proof}

An implementation of the above proof strategy, using as starting points the polytopes $P$ and $Q$ from \cref{ex:simple_slices_counterexample}, allows us to find an explicit example of a simple polytope with two slices that are not combinatorially equivalent but have isomorphic graphs. 
\begin{example}\label{ex:simple_counterexample}
    Let $R$ be the $5$-dimensional polytope with $120$ vertices provided in \cref{app:coordinates}. A standard implementation in a computer algebra software such as \texttt{SageMath} \cite{sagemath} or \texttt{Oscar.jl} \cite{OSCAR,OSCAR-book} shows that $R$ is simple. The slices of $R$ by the hyperplanes 
    \[
    H_1 = \set{x\in \R^5}{x_1+x_2+x_3+x_4+x_5=1}, \qquad H_2 = \set{x\in \R^5}{x_1+x_2+x_3+x_4-x_5=-15},
    \]
    produce the $4$-dimensional polytopes $P$ and $Q$ respectively. Using a computer algebra software, one can check that they are combinatorially equivalent to the polytopes in \cref{ex:simple_slices_counterexample}, hence
    \[
    P \not\cong Q, \qquad \text{and} \qquad G(P)\cong G(Q).
    \]
    Therefore, the slices of the simple polytope $R$ cannot be reconstructed from their graphs.
\end{example}

\begin{proof}[Proof of \cref{thm:slice_simple_polytope}]
    \cref{ex:simple_counterexample} provides an explicit example of a simple $5$-dimensional polytope with two $4$-dimensional slices that cannot be reconstructed from their vertex-edge graphs. This is the minimal-dimension case in which such a phenomenon can happen (because $3$-dimensional polytopes are always reconstructible). Moreover, for higher dimension $d$ it suffices to take $R_d = R\times \cube[d-5]$ to obtain a similar example in $\R^d$. This completes the proof.
\end{proof}

\printbibliography

\vfill
\noindent \textsc{Marie-Charlotte Brandenburg} \\
\textsc{Ruhr-Universität Bochum} \\
 \url{marie-charlotte.brandenburg@rub.de} \\

\noindent \textsc{Juliana Curtis} \\
\textsc{University of Washington} \\
 \url{jacca@uw.edu} \\

\noindent \textsc{Chiara Giardino} \\
\textsc{Ruhr-Universität Bochum} \\ 
\url{chiara.giardino@rub.de} \\

\noindent \textsc{Chiara Meroni} \\
\textsc{ETH Institute for Theoretical Studies} \\ 
\url{chiara.meroni@eth-its.ethz.ch} \\

\noindent \textsc{Arnau Padrol} \\
\textsc{Universitat de Barcelona \& Centre de Recerca Matem\`atica} \\ 
\url{arnau.padrol@ub.edu} \\

\newpage
\appendix

\section{Coordinates of the simple polytope with non-reconstructible slices}
\label{app:coordinates}

The $5$-dimensional polytope $R$ from \cref{ex:simple_counterexample} has $120$ vertices. Their coordinates are listed below.

\begingroup
\fontsize{8}{8}\selectfont
\setlength{\tabcolsep}{8pt}
\renewcommand{\arraystretch}{1.8}
\let\frac\tfrac

\begin{longtable}{
@{}r@{\quad}l@{\qquad}
r@{\quad}l@{\qquad}
r@{\quad}l@{}
}
\textbf{\#} & \textbf{vertex}
& \textbf{\#} & \textbf{vertex}
& \textbf{\#} & \textbf{vertex}\\
\hline
\endfirsthead

\textbf{\#} & \textbf{vertex}
& \textbf{\#} & \textbf{vertex}
& \textbf{\#} & \textbf{vertex}\\
\hline
\endhead

1 &
$\left(\frac12,0,\frac12,0,5\right)$
& 41 &
$\left(0,\frac{24}{43},\frac{19}{43},0,\frac{668}{43}\right)$
& 81 &
$\left(0,\frac7{29},0,\frac{12}{29},\frac5{29}\right)$
\\

2 &
$\left(\frac{14}{23},0,0,\frac9{23},\frac{348}{23}\right)$
& 42 &
$\left(0,1,0,0,\frac{91}{6}\right)$
& 82 &
$\left(\frac37,0,0,0,\frac{108}{7}\right)$
\\

3 &
$\left(\frac5{14},0,\frac1{14},\frac17,\frac37\right)$
& 43 &
$\left(0,\frac12,\frac38,0,\frac{11}{8}\right)$
& 83 &
$\left(0,\frac1{12},\frac14,\frac4{21},\frac{331}{21}\right)$
\\

4 &
$\left(\frac12,0,\frac1{22},0,\frac{10}{11}\right)$
& 44 &
$\left(0,\frac9{23},\frac8{23},0,\frac{16}{23}\right)$
& 84 &
$\left(\frac15,0,0,\frac4{105},\frac{328}{21}\right)$
\\

5 &
$\left(\frac17,\frac17,\frac17,\frac27,\frac{110}{7}\right)$
& 45 &
$\left(0,\frac35,\frac15,0,\frac65\right)$
& 85 &
$\left(0,0,0,0,15\right)$
\\

6 &
$\left(\frac14,0,0,\frac14,\frac{31}{2}\right)$
& 46 &
$\left(0,\frac37,\frac27,0,\frac9{14}\right)$
& 86 &
$\left(\frac3{25},\frac7{75},\frac4{25},\frac7{75},\frac{236}{15}\right)$
\\

7 &
$\left(0,0,0,1,\frac{348}{23}\right)$
& 47 &
$\left(0,\frac{24}{85},\frac{19}{85},\frac{42}{85},\frac{268}{17}\right)$
& 87 &
$\left(0,\frac1{21},0,0,\frac{326}{21}\right)$
\\

8 &
$\left(\frac{13}{25},0,0,0,\frac{22}{25}\right)$
& 48 &
$\left(0,\frac{24}{47},0,\frac{23}{47},\frac{732}{47}\right)$
& 88 &
$\left(0,\frac8{35},\frac15,0,\frac{110}{7}\right)$
\\

9 &
$\left(0,0,\frac{19}{36},\frac{17}{36},\frac{139}{9}\right)$
& 49 &
$\left(0,0,\frac7{36},\frac{19}{36},\frac5{36}\right)$
& 89 &
$\left(0,\frac4{21},\frac14,\frac1{12},\frac{331}{21}\right)$
\\

10 &
$\left(0,\frac4{33},\frac7{33},\frac{26}{99},\frac{1564}{99}\right)$
& 50 &
$\left(0,\frac15,\frac15,\frac25,\frac{79}{5}\right)$
& 90 &
$\left(0,0,\frac16,\frac{25}{126},\frac{988}{63}\right)$
\\

11 &
$\left(1,0,0,0,10\right)$
& 51 &
$\left(0,0,\frac{19}{37},\frac{18}{37},\frac{572}{37}\right)$
& 91 &
$\left(0,0,0,\frac1{21},\frac{326}{21}\right)$
\\

12 &
$\left(\frac8{29},0,\frac3{29},\frac{18}{29},\frac{444}{29}\right)$
& 52 &
$\left(0,0,\frac3{29},\frac{26}{29},\frac{444}{29}\right)$
& 92 &
$\left(0,0,0,0,\frac94\right)$
\\

13 &
$\left(1,0,0,0,\frac{552}{37}\right)$
& 53 &
$\left(0,0,\frac{11}{37},\frac{26}{37},\frac{572}{37}\right)$
& 93 &
$\left(\frac5{24},\frac1{12},0,\frac7{24},\frac5{24}\right)$
\\

14 &
$\left(\frac{38}{51},\frac5{51},0,\frac8{51},\frac{776}{51}\right)$
& 54 &
$\left(0,0,0,1,5\right)$
& 94 &
$\left(\frac16,\frac16,0,\frac14,\frac5{24}\right)$
\\

15 &
$\left(\frac{18}{19},0,\frac1{19},0,\frac{284}{19}\right)$
& 55 &
$\left(0,1,0,0,5\right)$
& 95 &
$\left(0,\frac13,0,\frac14,\frac5{24}\right)$
\\

16 &
$\left(\frac{29}{48},\frac5{48},0,\frac1{12},\frac{123}{8}\right)$
& 56 &
$\left(0,0,\frac12,\frac{19}{42},\frac{325}{21}\right)$
& 96 &
$\left(0,\frac7{48},\frac38,\frac1{16},\frac5{24}\right)$
\\

17 &
$\left(\frac{19}{28},\frac5{28},0,\frac17,\frac{107}{7}\right)$
& 57 &
$\left(0,\frac{17}{36},\frac{19}{36},0,\frac{139}{9}\right)$
& 97 &
$\left(\frac{17}{72},\frac1{36},\frac5{36},\frac{13}{72},\frac5{24}\right)$
\\

18 &
$\left(\frac{38}{39},\frac1{39},0,0,\frac{584}{39}\right)$
& 58 &
$\left(\frac16,0,0,\frac5{54},\frac{422}{27}\right)$
& 98 &
$\left(\frac16,\frac16,\frac5{48},\frac7{48},\frac5{24}\right)$
\\

19 &
$\left(\frac3{11},\frac3{11},\frac1{11},0,\frac9{11}\right)$
& 59 &
$\left(\frac15,\frac1{20},0,\frac1{12},\frac{47}{3}\right)$
& 99 &
$\left(0,\frac13,\frac3{16},\frac1{16},\frac5{24}\right)$
\\

20 &
$\left(\frac5{18},\frac29,\frac16,0,\frac89\right)$
& 60 &
$\left(\frac4{21},\frac4{63},\frac1{63},\frac2{21},\frac{988}{63}\right)$
& 100 &
$\left(\frac5{24},0,0,\frac38,\frac5{24}\right)$
\\

21 &
$\left(\frac{19}{24},\frac5{24},0,0,\frac{91}{6}\right)$
& 61 &
$\left(0,0,0,\frac12,\frac{31}{2}\right)$
& 101 &
$\left(0,\frac12,0,0,\frac12\right)$
\\

22 &
$\left(\frac8{53},\frac8{53},\frac{11}{53},\frac{26}{53},\frac{828}{53}\right)$
& 62 &
$\left(0,\frac1{20},0,\frac{17}{60},\frac{47}{3}\right)$
& 102 &
$\left(\frac14,0,\frac1{24},\frac7{24},\frac5{24}\right)$
\\

23 &
$\left(\frac{11}{111},\frac{16}{111},\frac9{37},\frac{19}{37},\frac{1736}{111}\right)$
& 63 &
$\left(0,0,0,\frac7{27},\frac{422}{27}\right)$
& 103 &
$\left(0,\frac7{18},\frac16,0,\frac{11}{36}\right)$
\\

24 &
$\left(\frac13,\frac19,0,\frac19,\frac49\right)$
& 64 &
$\left(0,\frac13,\frac13,0,\frac{47}{3}\right)$
& 104 &
$\left(0,\frac5{36},\frac5{12},0,\frac{11}{36}\right)$
\\

25 &
$\left(\frac5{14},\frac3{14},0,0,\frac{11}{14}\right)$
& 65 &
$\left(0,0,\frac17,\frac29,\frac{988}{63}\right)$
& 105 &
$\left(\frac27,0,\frac3{28},\frac5{28},\frac14\right)$
\\

26 &
$\left(0,0,\frac12,\frac12,\frac52\right)$
& 66 &
$\left(\frac{41}{169},\frac7{169},\frac1{169},0,\frac{2644}{169}\right)$
& 106 &
$\left(\frac4{11},0,\frac1{44},\frac7{44},\frac{15}{44}\right)$
\\

27 &
$\left(\frac{37}{60},\frac1{12},0,0,\frac{461}{30}\right)$
& 67 &
$\left(0,\frac{16}{99},\frac{13}{99},\frac{10}{33},\frac{1564}{99}\right)$
& 107 &
$\left(\frac14,0,\frac14,0,\frac12\right)$
\\

28 &
$\left(0,0,1,0,5\right)$
& 68 &
$\left(0,\frac{16}{107},\frac{21}{107},\frac{30}{107},\frac{1692}{107}\right)$
& 108 &
$\left(\frac5{11},0,0,0,\frac{29}{44}\right)$
\\

29 &
$\left(0,\frac12,\frac12,0,\frac52\right)$
& 69 &
$\left(\frac9{37},\frac4{111},0,0,\frac{1736}{111}\right)$
& 109 &
$\left(\frac{25}{64},0,0,\frac9{64},\frac{25}{64}\right)$
\\

30 &
$\left(0,0,\frac14,\frac12,\frac14\right)$
& 70 &
$\left(0,\frac{16}{73},0,\frac{17}{73},\frac{1148}{73}\right)$
& 110 &
$\left(\frac7{22},\frac1{11},0,\frac3{22},\frac{15}{44}\right)$
\\

31 &
$\left(\frac{55}{224},\frac9{224},0,0,\frac{219}{14}\right)$
& 71 &
$\left(0,\frac27,0,0,\frac{219}{14}\right)$
& 111 &
$\left(\frac{17}{48},\frac1{24},0,\frac7{48},\frac{17}{48}\right)$
\\

32 &
$\left(0,0,\frac14,\frac{17}{36},\frac5{36}\right)$
& 72 &
$\left(0,0,1,0,6\right)$
& 112 &
$\left(\frac14,\frac14,0,0,\frac12\right)$
\\

33 &
$\left(\frac{38}{181},\frac{10}{181},0,\frac{13}{181},\frac{2836}{181}\right)$
& 73 &
$\left(\frac4{95},\frac{12}{95},\frac{17}{95},\frac{22}{95},\frac{300}{19}\right)$
& 113 &
$\left(0,0,\frac12,0,\frac12\right)$
\\

34 &
$\left(0,\frac34,0,0,\frac32\right)$
& 74 &
$\left(\frac5{13},\frac1{13},\frac1{13},0,\frac{202}{13}\right)$
& 114 &
$\left(\frac14,0,\frac18,\frac5{24},\frac5{24}\right)$
\\

35 &
$\left(0,\frac13,0,\frac13,\frac13\right)$
& 75 &
$\left(0,0,0,\frac{19}{29},\frac5{29}\right)$
& 115 &
$\left(0,0,\frac38,\frac5{24},\frac5{24}\right)$
\\

36 &
$\left(0,\frac{47}{379},\frac{91}{379},\frac{131}{379},\frac{5974}{379}\right)$
& 76 &
$\left(\frac{17}{19},0,\frac2{19},0,\frac{284}{19}\right)$
& 116 &
$\left(0,0,\frac1{15},\frac{31}{60},\frac5{24}\right)$
\\

37 &
$\left(0,\frac3{16},\frac{19}{72},\frac{79}{144},\frac{283}{18}\right)$
& 77 &
$\left(0,\frac27,\frac{13}{56},\frac18,\frac5{28}\right)$
& 117 &
$\left(0,0,0,\frac{39}{64},\frac{25}{128}\right)$
\\

38 &
$\left(0,\frac5{89},\frac{27}{89},\frac{57}{89},\frac{1384}{89}\right)$
& 78 &
$\left(0,\frac{19}{89},\frac{27}{89},\frac{13}{89},\frac{15}{89}\right)$
& 118 &
$\left(0,0,0,\frac14,\frac{11}{8}\right)$
\\

39 &
$\left(0,\frac{13}{150},\frac{19}{50},\frac8{15},\frac{78}{5}\right)$
& 79 &
$\left(\frac{19}{23},0,\frac2{23},0,\frac{346}{23}\right)$
& 119 &
$\left(0,\frac1{15},0,\frac{31}{60},\frac5{24}\right)$
\\

40 &
$\left(0,\frac{16}{77},\frac{19}{77},\frac6{11},\frac{1212}{77}\right)$
& 80 &
$\left(0,\frac27,0,\frac5{14},\frac5{28}\right)$
& 120 &
$\left(\frac1{18},0,0,\frac{19}{36},\frac5{24}\right)$
\\

\end{longtable}
\endgroup

~\\
The $4$-dimensional slice $P$ cut out by the hyperplane $x_1+x_2+x_3+x_4+x_5=1$ contains the vertices of the original polytope indexed by 
\[
3,\ 24,\ 30,\ 35,\ 101,\ 107,\ 112,\ 113,
\]
and intersects the relative interiors of the edges 
\[
[8,109],\ [44,78],\ [44,104],\ [46,77],\ [46,103],\ [54,75],\ [108,109],\ [116,118],\ [117,118],\ [118,119],\ [118,120].
\]

The $4$-dimensional slice $Q$ cut out by the hyperplane $x_1+x_2+x_3+x_4-x_5=-15$ contains the vertices of the original polytope indexed by 
\[
5,\ 6,\ 50,\ 61,\ 64,\ 74,\ 82,\ 85
\]
and intersects the relative interiors of the edges 
\[
[16,33],\ [23,36],\ [27,31],\ [36,37],\ [36,38],\ [36,39],\ [42,71],\ [48,70],\ [53,65],\ [56,83],\ [56,90].
\]

~\\

To facilitate the reproducibility of the computations, we provide below the coordinates of the $120$ vertices of the $5$-dimensional polytope in a format that can be used directly as input for a computer algebra system:\\
~\\
{\scriptsize \texttt{[[1/2, 0, 1/2, 0, 5], [14/23, 0, 0, 9/23, 348/23], [5/14, 0, 1/14, 1/7, 3/7], [1/2, 0, 1/22, 0, 10/11], [1/7, 1/7, 1/7, 2/7, 110/7], [1/4, 0, 0, 1/4, 31/2], [0, 0, 0, 1, 348/23], [13/25, 0, 0, 0, 22/25], [0, 0, 19/36, 17/36, 139/9], [0, 4/33, 7/33, 26/99, 1564/99], [1, 0, 0, 0, 10], [8/29, 0, 3/29, 18/29, 444/29], [1, 0, 0, 0, 552/37], [38/51, 5/51, 0, 8/51, 776/51], [18/19, 0, 1/19, 0, 284/19], [29/48, 5/48, 0, 1/12, 123/8], [19/28, 5/28, 0, 1/7, 107/7], [38/39, 1/39, 0, 0, 584/39], [3/11, 3/11, 1/11, 0, 9/11], [5/18, 2/9, 1/6, 0, 8/9], [19/24, 5/24, 0, 0, 91/6], [8/53, 8/53, 11/53, 26/53, 828/53], [11/111, 16/111, 9/37, 19/37, 1736/111], [1/3, 1/9, 0, 1/9, 4/9], [5/14, 3/14, 0, 0, 11/14], [0, 0, 1/2, 1/2, 5/2], [37/60, 1/12, 0, 0, 461/30], [0, 0, 1, 0, 5], [0, 1/2, 1/2, 0, 5/2], [0, 0, 1/4, 1/2, 1/4], [55/224, 9/224, 0, 0, 219/14], [0, 0, 1/4, 17/36, 5/36], [38/181, 10/181, 0, 13/181, 2836/181], [0, 3/4, 0, 0, 3/2], [0, 1/3, 0, 1/3, 1/3], [0, 47/379, 91/379, 131/379, 5974/379], [0, 3/16, 19/72, 79/144, 283/18], [0, 5/89, 27/89, 57/89, 1384/89], [0, 13/150, 19/50, 8/15, 78/5], [0, 16/77, 19/77, 6/11, 1212/77], [0, 24/43, 19/43, 0, 668/43], [0, 1, 0, 0, 91/6], [0, 1/2, 3/8, 0, 11/8], [0, 9/23, 8/23, 0, 16/23], [0, 3/5, 1/5, 0, 6/5], [0, 3/7, 2/7, 0, 9/14], [0, 24/85, 19/85, 42/85, 268/17], [0, 24/47, 0, 23/47, 732/47], [0, 0, 7/36, 19/36, 5/36], [0, 1/5, 1/5, 2/5, 79/5], [0, 0, 19/37, 18/37, 572/37], [0, 0, 3/29, 26/29, 444/29], [0, 0, 11/37, 26/37, 572/37], [0, 0, 0, 1, 5], [0, 1, 0, 0, 5], [0, 0, 1/2, 19/42, 325/21], [0, 17/36, 19/36, 0, 139/9], [1/6, 0, 0, 5/54, 422/27], [1/5, 1/20, 0, 1/12, 47/3], [4/21, 4/63, 1/63, 2/21, 988/63], [0, 0, 0, 1/2, 31/2], [0, 1/20, 0, 17/60, 47/3], [0, 0, 0, 7/27, 422/27], [0, 1/3, 1/3, 0, 47/3], [0, 0, 1/7, 2/9, 988/63], [41/169, 7/169, 1/169, 0, 2644/169], [0, 16/99, 13/99, 10/33, 1564/99], [0, 16/107, 21/107, 30/107, 1692/107], [9/37, 4/111, 0, 0, 1736/111], [0, 16/73, 0, 17/73, 1148/73], [0, 2/7, 0, 0, 219/14], [0, 0, 1, 0, 6], [4/95, 12/95, 17/95, 22/95, 300/19], [5/13, 1/13, 1/13, 0, 202/13], [0, 0, 0, 19/29, 5/29], [17/19, 0, 2/19, 0, 284/19], [0, 2/7, 13/56, 1/8, 5/28], [0, 19/89, 27/89, 13/89, 15/89], [19/23, 0, 2/23, 0, 346/23], [0, 2/7, 0, 5/14, 5/28], [0, 7/29, 0, 12/29, 5/29], [3/7, 0, 0, 0, 108/7], [0, 1/12, 1/4, 4/21, 331/21], [1/5, 0, 0, 4/105, 328/21], [0, 0, 0, 0, 15], [3/25, 7/75, 4/25, 7/75, 236/15], [0, 1/21, 0, 0, 326/21], [0, 8/35, 1/5, 0, 110/7], [0, 4/21, 1/4, 1/12, 331/21], [0, 0, 1/6, 25/126, 988/63], [0, 0, 0, 1/21, 326/21], [0, 0, 0, 0, 9/4], [5/24, 1/12, 0, 7/24, 5/24], [1/6, 1/6, 0, 1/4, 5/24], [0, 1/3, 0, 1/4, 5/24], [0, 7/48, 3/8, 1/16, 5/24], [17/72, 1/36, 5/36, 13/72, 5/24], [1/6, 1/6, 5/48, 7/48, 5/24], [0, 1/3, 3/16, 1/16, 5/24], [5/24, 0, 0, 3/8, 5/24], [0, 1/2, 0, 0, 1/2], [1/4, 0, 1/24, 7/24, 5/24], [0, 7/18, 1/6, 0, 11/36], [0, 5/36, 5/12, 0, 11/36], [2/7, 0, 3/28, 5/28, 1/4], [4/11, 0, 1/44, 7/44, 15/44], [1/4, 0, 1/4, 0, 1/2], [5/11, 0, 0, 0, 29/44], [25/64, 0, 0, 9/64, 25/64], [7/22, 1/11, 0, 3/22, 15/44], [17/48, 1/24, 0, 7/48, 17/48], [1/4, 1/4, 0, 0, 1/2], [0, 0, 1/2, 0, 1/2], [1/4, 0, 1/8, 5/24, 5/24], [0, 0, 3/8, 5/24, 5/24], [0, 0, 1/15, 31/60, 5/24], [0, 0, 0, 39/64, 25/128], [0, 0, 0, 1/4, 11/8], [0, 1/15, 0, 31/60, 5/24], [1/18, 0, 0, 19/36, 5/24]]}}

\end{document}